\documentclass[12pt]{article} 
\usepackage{authblk}
\usepackage{graphicx}
\usepackage{latexsym,amsfonts,amssymb}
\usepackage{amsmath, amsthm}
\usepackage[dvips]{color}
\usepackage{colordvi,multicol}
\usepackage[marginal]{footmisc}
\usepackage{mathrsfs}

\usepackage{hyperref}
\usepackage{marvosym}
\newtheorem{thm}{Theorem}[section]

\newtheorem*{thm*}{Theorem}
\newtheorem{cor}[thm]{Corollary}
\newtheorem{lem}[thm]{Lemma}
\newtheorem{pro}[thm]{Proposition}

\newtheorem{rem}[thm]{Remark}

\numberwithin{equation}{section}
\date{}

\begin{document}
	
	\title{\bf Exponential mixing for Hamiltonian shear flow}
    \author[1,2]{Weili Zhang\thanks{zhangweili@amss.ac.cn}}
    \affil[1]{Beijing Institute of Mathematical Sciences and Applications, Beijing 101408, China}
    \affil[2]{Yau Mathematical Sciences Center, Tsinghua University, Beijing 100084, China}
    \renewcommand\Authfont{\normalsize}  
    \renewcommand\Affilfont{\small}             
    \maketitle
	\noindent{\bf Abstract:} We consider the advection equation on $\mathbb{T}^2$ with a real analytic and time-periodic velocity field that alternates between two Hamiltonian shears. Randomness is injected by alternating the vector field randomly in time between just two distinct shears. We prove that, under general conditions, these models have a positive top Lyapunov exponent and exhibit exponential mixing. This framework is then applied to the Pierrehumbert model with randomized time and to a model analogous to the Chirikov standard map.
	
	\section{Introduction}
	\quad Mixing by incompressible flows is a fundamental problem in fluid mechanics with relevance to both engineering and the natural sciences.  In its simplest mathematical setting, this phenomenon is relevant to the study of the long time dynamics of the advection equation. Given a divergence-free and time-dependent velocity field $b:[0,\infty)\times\mathbb{T}^2\to\mathbb{R}^2$, where $\mathbb{T}^2:=\mathbb{R}^2/(2\pi \mathbb{Z})^2$, consider the advection equation on $\mathbb{T}^2$
	\begin{equation}\label{advection equation}
		\begin{cases}
			\partial_t u(t,x)+b(t,x)\cdot \nabla u(t,x) = 0, \\
			u(0,x) = u_0(x).
		\end{cases}
	\end{equation}
	We are interested in how effectively some mean-zero initial data $u_0$ is mixed when advected by $b$. We can measure mixing by the geometric mixing scale in Bressan's sense. More precisely, given a vector field $b\in L^1([0,1]; W^{1,1}(\mathbb{T}^2;\mathbb{R}^2))$, the \textbf{mixing scale} $\text{mix}(b)$ is defined by 
	\begin{equation*}
		\text{mix}(b):=\sup\left\{r(B):\left|\frac{1}{|B|}\int_B u(1,x)dx\right|>1\right\},
	\end{equation*}
	where the supremum is over all balls $B\subseteq \mathbb{T}^2$, $r(B)$ is the radius of $B$ and $u(t,x)$ is the solution to the advection equation (\ref{advection equation}) with initial data $u_0\in L^\infty(\mathbb{T}^{2})$ with mean zero. 
		
	For a given divergence-free velocity field $b$, let $\Psi(t,\cdot) :\mathbb{T}^2\to\mathbb{T}^2$ denote the associated flow map, which solves 
	\begin{equation}\label{Lagrangian equation}
		\begin{cases}
			\frac{d}{dt} \Psi(t, x) = b(t,\Psi(t, x)), \\
			\Psi(0,x) = x.
		\end{cases}
	\end{equation}
	For any Lipschitz $b$, $\Psi(t,\cdot)$ is well defined and is a Lebesgue measure-preserving homeomorphism for each $t\geq 0$. The solution to the pure advection equation (\ref{advection equation}) is then given by 
	\begin{equation*}
		u(t,x)=u_0(\Psi^{-1}(t,x)).
	\end{equation*}
	From the above equation, one can naturally quantify mixing in terms of correlation decay of sufficiently regular observables for the dynamical system on $\mathbb{T}^2$ generated by $\Psi(t,\cdot)$.
	
	Our main goal in this paper is to introduce a time-periodic velocity field that is analytic and whose dynamics we can understand deeply enough to establish exponential mixing. We consider in particular alternating shear field. Let $X_i: \mathbb{T}^2\to \mathbb{R}^2$, $i=1,2$ be the Hamiltonian horizontal and vertical shears, respectively, defined by
	\begin{equation*}
		X_1(x):=(f_1(p),0)\quad \text{and} \quad X_2(x):=(0, f_2(q)),\ \ x=(q,p)
	\end{equation*}
	where $f_i \in C^{\omega}(\mathbb{S}^1, \mathbb{R})$, $i=1,2$ are non-constant, and $\mathbb{S}^1:= \mathbb{R}/(2\pi \mathbb{Z})$. Fix a sufficiently large $T>0$, we then define the divergence-free velocity field $b_{\underline{\tau}}:[0,\infty)\times\mathbb{T}^2\to\mathbb{R}^2$ by
	\[
	b_{\underline{\tau}}(t,x) = 
	\begin{cases} 
		-\tau_{\lceil t \rceil}X_2(x) & \text{if} \ \lceil t \rceil \ \text{is odd}, \\
		-\tau_{\lceil t \rceil}X_1(x) & \text{if} \ \lceil t \rceil \ \text{is even}, 
	\end{cases}
	\]
	where $\underline{\tau}=(\tau_1,\tau_2,\cdots)$ be chosen independently and identically distributed with respect to the uniform distribution on $[0,T]$. Thus our sample space can be written as $[0,T]^\mathbb{N}$, equipped with the probability $\mathbb{P}$. We write $u_{\underline{\tau}}$ to denote the solution to (\ref{advection equation}) with $b=b_{\underline{\tau}}$, and we write $\Psi_{\underline{\tau}}$ to denote the solution to (\ref{Lagrangian equation}) with $b=b_{\underline{\tau}}$.
	
	Denote $F:=\{(q,p):f_1(p)=0 \text{ and } f_2(q)=0\}$. Indeed, $b_{\underline{\tau}}(t,x)=0$ for any $t>0$ and $x\in F$. To deal with this, we consider compositions of random maps on $\mathcal{O}:=\mathbb{T}^2\setminus F$
	\begin{align*}
		\Phi_{\underline{\tau}}^m=
		\begin{cases}
			I, & m=0;\\
			\varphi_{\tau_{2m}}^{(2)}\circ\varphi_{\tau_{2m-1}}^{(1)}(\Phi_{\underline{\tau}}^{m-1}),& m=1,2,\cdots,
		\end{cases}
	\end{align*}    
	where $I$ is the identity and
	\[ 
	\varphi_{\tau}^{(1)}(x) = \begin{pmatrix}
		q+\tau f_1(p) \\
		p
	\end{pmatrix},\quad
	\varphi_{\tau}^{(2)}(x) = \begin{pmatrix}
		q \\
		p+\tau f_2(q)
	\end{pmatrix}.
	\]
	It follows from the independence of the $\tau_i$ that $\{\Phi_{\underline{\tau}}^m\}_{m=0}^{\infty}$ is a Markov chain on $\mathcal{O}$, which we refer to as the one-point Markov chain. Its transition kernel $P:\mathcal{O}\times\mathscr{B}(\mathcal{O})\to [0,1]$ is defined by
	\begin{equation*}
		P(x,A)=\mathbb{P}(\Phi_{\underline{\tau}}(x)\in A)=\mathbb{E}\chi_A(\Phi_{\underline{\tau}}(x))
	\end{equation*}
	where $\mathscr{B}(\mathcal{O})$ is the Borel $\sigma$-algebra on $\mathcal{O}$ and $\chi_A$ is the indicator function of $A$. $P$ acts on measurable functions $g:\mathcal{O}\to \mathbb{R}$ by
	\begin{equation*}
		Pg(x)=\mathbb{E}g(\Phi_{\underline{\tau}}(x)).
	\end{equation*} 
	And $P$ acts on measure $\mu$ on $\mathcal{O}$ by 
	\begin{equation*}
		P^*\mu(A):=\int P(x,A)\mu(dx),  \ \ A\in\mathscr{B}(\mathcal{O}).
	\end{equation*}
	The stationary measure plays an important role in our results on ergodicity and chaos, recall a probability measure $\mu$ on $\mathcal{O}$ is \textbf{stationary} for the Markov chain $\{\Phi_{\underline{\tau}}^m\}$ if $P^*\mu=\mu$. We say that $P$ is \textbf{uniquely ergodic} if the set of stationary measures has cardinality one. Given a function $V: \mathcal{O}\to [1,\infty)$, we define the weighted norm
	\[
	\|g\|_V:=\sup_{x\in \mathcal{O}}\frac{|g(x)|}{V(x)},
	\]
	and define $\mathcal{M}_V(\mathcal{O})$ to be the space of measurable observables $g: \mathcal{O}\to\mathbb{R}$ such that $\|g\|_V<\infty$. We say that $P$ is $V$-\textbf{uniformly geometrically ergodic}, if $P$ admits a unique stationary measure $\mu$ on $\mathcal{O}$, and there exist constants $C>0$ and $\gamma\in (0,1)$ such that the bound 
	\begin{equation*}
		\left |P^m g(x)-\int g d\mu\right |\leq CV(x)\|g\|_V\gamma^m,
	\end{equation*}
	holds for all $x\in \mathcal{O}$ and $g\in \mathcal{M}_V(\mathcal{O})$. 
	Denote $\mathcal{X}=\{X_1, X_2\}$, $C_{f_i}:=\{z\in\mathbb{S}^1:f_i(z)=0\}$, $C_{f_i'}:=\{z\in\mathbb{S}^1:f_i'(z)=0\}$. Assume 
	\begin{itemize}
		\item [(H1)] $C_{f_i}\cap C_{f_i'}=\emptyset$, $i=1, 2$.
	\end{itemize}
	\begin{thm}[\textbf{Uniformly geometrically ergodic}]\label{Uniformly geometrically ergodic}
		Let $\{\Phi_{\underline{\tau}}^m\}$ be as above. Assume $(H1)$ holds. Then there exists a continuous function $V:\mathcal{O}\to [1,\infty)$ such that the transition kernel $P$ of $\{\Phi_{\underline{\tau}}^m\}$ is $V$-uniformly geometrically ergodic.
	\end{thm}
	
	The lift of $\{\Phi_{\underline{\tau}}^m\}$ to the tangent bundle $T\mathcal{O}$ is the Markov chain $\{\hat{\Phi}_{\underline{\tau}}^m\}$ with transition kernel $\hat{P}((x,u),\hat{A})=\mathbb{P}(\hat{\Phi}_{\underline{\tau}}(x,u)\in\hat{A})$, its corresponding vector fields are $\hat{\mathcal{X}}=\{\hat{X_1}, \hat{X_2}\}$, where for $(x,u)\in T\mathcal{O}$,
	\begin{equation*}
		\hat{\Phi}_{\underline{\tau}}^m(x,u)=\left(\Phi_{\underline{\tau}}^m(x), D_x\Phi_{\underline{\tau}}^m u\right),
	\end{equation*}
	and
	\[ 
	\hat{X_i}(x,u) = \begin{pmatrix}
		X_i(x) \\
		DX_i(x)u
	\end{pmatrix},\quad 
    i=1, 2.
	\]
	Assume
	\begin{itemize}
		\item[(H2)] $\text{Lie}_{\hat{x}}(\hat{\mathcal{X}})=T_{\hat{x}}T\mathcal{O}$ for some ${\hat{x}}\in T\mathcal{O}$.
	\end{itemize}
	
	\begin{thm}[\textbf{Positivity of the top Lyapunov exponent}]\label{Positivity of the top Lyapunov exponent}
		Let $\{\Phi_{\underline{\tau}}^m\}$ be as above. Assume $(H1)$ and $(H2)$ hold. Then the random dynamical system $\{\Phi_{\underline{\tau}}^m\}$ almost surely has a positive top Lyapunov exponent 
		\begin{equation*}
			\lambda_1:=\lim_{m\to\infty}\frac{1}{m}\log \|D_x\Phi_{\underline{\tau}}^m(x)\|
		\end{equation*}
		for $\mu$-a.e. $x\in\mathcal{O}$.
	\end{thm}
	 The two-point chain is the Markov chain $\{\tilde{\Phi}_{\underline{\tau}}^m\}$ on $\mathcal{O}^{(2)}:=\mathcal{O}\times \mathcal{O}\setminus \Delta$ ($\Delta$ is the invariant set of $\tilde{\Phi}_{\underline{\tau}}$, i.e., $\tilde{\Phi}_{\underline{\tau}}\Delta=\Delta$ for every $\underline{\tau}$) with transition kernel $P^{(2)}((x,y),A^{(2)})=\mathbb{P}(\tilde{\Phi}_{\underline{\tau}}(x,y)\in A^{(2)})$, its corresponding vector fields are $\tilde{\mathcal{X}}=\{\tilde{X_1}, \tilde{X_2}\}$, where for $(x,y)\in\mathcal{O}^{(2)}$,
	 \begin{equation*}
	 	\tilde{\Phi}_{\underline{\tau}}^m(x,y)=(\Phi_{\underline{\tau}}^m(x), \Phi_{\underline{\tau}}^m(y)),
	 \end{equation*}
	 and 
	 \[ 
	 \tilde{X_i}(x,y) = \begin{pmatrix}
	 	X_i(x) \\
	 	X_i(y)
	 \end{pmatrix},\quad
	 i=1, 2.
	 \] 
	 Assume
	\begin{itemize}
		\item[(H3)] $\text{Lie}_{\tilde{x}}(\tilde{\mathcal{X}})=T_{\tilde{x}}\mathcal{O}^{(2)}$ for some ${\tilde{x}}\in \mathcal{O}^{(2)}$.
	\end{itemize}
    To make our notation match with the mixing scale, we write $b_{\underline{\tau}}^\varrho(t,x)=\varrho b_{\underline{\tau}}(\varrho t,x)$ to denote the speeding time up by a factor of $\varrho$, and we write $u_{\underline{\tau}}$ to denote the solution to (\ref{advection equation}) with $b=b_{\underline{\tau}}$.
	\begin{thm}[\textbf{Exponential mixing}]\label{Exponential mixing}
		Let $b_{\underline{\tau}}^\varrho$ be as above. Assume $(H1)$, $(H2)$ and $(H3)$ hold. Then there is a random variable $\xi$ which is positive almost surely, such that 
		\[
		|\log \text{mix}(b_{\underline{\tau}}^\varrho)|\geq \xi(\underline{\tau})\|D_xb_{\underline{\tau}}^\varrho\|_{L^1([0,1]\times \mathbb{T}^{2})},
		\]
		for any integer $\varrho>0$ almost surely.
	\end{thm}
	The Pierrehumbert model with randomized time is defined by setting
	\begin{equation*}
		f_1(p)=\sin(p)\quad \text{and} \quad f_2(q)=\sin(q).
	\end{equation*}
	\begin{cor}[\textbf{The Pierrehumbert model with randomized time}]\label{The Pierrehumbert model with randomized time}
		Let $\{\Phi_{\underline{\tau}}^m\}$, and $b_{\underline{\tau}}^\varrho$ be as above with $f_1(p):=\sin(p)$ and $f_2(q):=\sin(q)$. Then Theorem \ref{Uniformly geometrically ergodic}, Theorem \ref{Positivity of the top Lyapunov exponent}, and Theorem \ref{Exponential mixing} hold for the Pierrehumbert model with randomized time.
	\end{cor}
    Let $\{\Phi_{\underline{\tau}}^m\}$ be as above. Take $f_1(p)=\sin p$ and $f_2(q)=q$. Then 
    \[
    \Phi_{\underline{\tau}}=(q+\tau_1\sin p, p+\tau_2(q+\tau_1\sin p))
    \]
    serves as an analog to the Chirikov standard map, which is an area-preserving map on $\mathbb{T}^2$ defined by
    \[
    \Psi(q,p)=(q+K\sin p, p+q+K\sin p),
    \]
    where $K$ is a parameter in $\mathbb{R}$. 
    \begin{cor}[\textbf{Analog model to the Chirikov standard map}]\label{Analog model to the Chirikov standard map}
    	Let $\{\Phi_{\underline{\tau}}^m\}$, and $b_{\underline{\tau}}^\varrho$ be as above with $f_1(p):=\sin(p)$ and $f_2(q):=q$. Then, for the model analogous to the Chirikov standard map, Theorem \ref{Uniformly geometrically ergodic}, Theorem \ref{Positivity of the top Lyapunov exponent}, and Theorem \ref{Exponential mixing} apply.
    \end{cor}
    \begin{rem}
    	Note that the function $f_2: \mathbb{S}^1\to\mathbb{R}$, $q\mapsto q$, is not a continuous function. However, we can regard it as the identity map from $\mathbb{S}^1$ to $\mathbb{S}^1$, which is obviously analytic. And the results in this paper still apply in this setting.
    \end{rem}
	\textbf{Relation to existing results.} Closest to the present work are the results of Agazzi, Mattingly, and Melikechi \cite{AMM, AMM2}, of Blumenthal, Coti Zelati, and Gvalani \cite{BCZG} and of Cooperman \cite{Co}. In \cite{AMM, AMM2}, Agazzi, Mattingly, and Melikechi study random splitting and apply their results to random splittings of fluid models. They prove these random splittings are ergodic and converge to their deterministic counterparts in a certain sense, and, for conservative Lorenz-96 and $2d$ Euler, that their top Lyapunov exponent is positive. In \cite{BCZG}, Blumenthal, Coti Zelati, and Gvalani introduce a robust framework for establishing exponential mixing in random dynamical systems. They subsequently apply this framework to the 1994 Pierrehumbert model \cite{P}, which consists of alternating periodic shear flows with randomized phases. In \cite{Co}, Cooperman employs the techniques developed in \cite{BCZG} and further addresses the challenges posed by the presence of almost surely stationary points in the flow. This allows him to demonstrate exponential mixing in Bressan's sense for the Pierrehumbert model under conditions where the switching times between horizontal and vertical shearing are randomized. 
	
	For comparison, drawing inspiration from \cite{AMM, AMM2}, we conceptualize our model as a random splitting of Hamiltonian horizontal and vertical shear flows. We establish the existence of a unique stationary measure and employ the framework of \cite{B} to derive sufficient conditions under which the random dynamical systems induced by horizontal and vertical shear flows exhibit a positive top Lyapunov exponent. Furthermore, we prove the exponential mixing properties of the corresponding advection equation, guided by the insights from \cite{BCZG, Co}. Although many of the basic ideas are the same, there are several significant differences. For example, we investigate a more general class of Hamiltonian shear flows, , which necessitates overcoming the challenges introduced by the presence of an invariant set in the two-point Markov chain. Using this framework, we establish the exponential mixing for passive scalar advection under alternating sine shear flows with random time, while Pierrehumbert \cite{P} studied with random phase. Moreover, we establish the exponential mixing for an analog model to the Chirikov standard map, further extending the applicability of our approach.
	
	The literature on Lyapunov exponents is extensive and spans various domains. In deterministic systems, establishing that the top Lyapunov exponent is positive is notoriously challenging, as discussed in \cite{PC, W, Y}. Even in the context of random systems, effective methods are relatively scarce. Furstenberg's seminal work in 1963 \cite{F} initiated extensive research into criteria for the Lyapunov exponent of random matrix products, leading to a substantial body of literature following Furstenberg's approach, including works such as \cite{AV, B, BL, C, L, R, V}. In this manuscript, we particularly draw on the methodologies from \cite{B}. Additionally, we acknowledge significant results regarding positive Lyapunov exponents in various fluid models, including Lagrangian flows advected by stochastic 2D Navier-Stokes equations and Galerkin truncations of these equations, as discussed in \cite{BP, BBP}. Furthermore, we note studies providing quantitative estimates of the top Lyapunov exponent, such as \cite{BXY, BBP22}. 
	
	Passive scalar mixing by fluid motion is a critical aspect of many phenomena, as discussed in \cite{P99, SS, W00, WF}. Understanding mixing is a central question in both physics and engineering applications and has been a topic of active research in mathematics, for example, \cite{BBP2021, BBP2022, CZDE, CLS, FI, LTD, SC} and the references therein. In particular, a number of recent works have addressed Bressan's mixing conjecture \cite{Br}, see, e.g., \cite{ACM, EZ, ZY, C, CDL, IKX}. The mixing properties of the Pierrehumbert model have been extensively studied in the applied and computational literature, see, e.g., \cite{PY, T, TDG}. Numerical evidence suggests that it is a universal exponential mixer \cite{CRWZ}, with the first rigorous mathematical proof provided in \cite{BCZG}. This work inspired the proof of exponential mixing for alternating sine shears via random durations in \cite{C}.
	
	This paper is organized as follows. Section 2 contains some preliminary results for Markov chains and the proof of Uniform geometric ergodicity of the one-point Markov chain. In Section 3, we demonstrate the positivity of the top Lyapunov exponent for the model that alternates between two Hamiltonian shears. Section 4 contains the proof of exponential mixing in the sense of Bressan, while Section 5 extends the abstract results to applications in the Pierrehumbert model with randomized timing and a model analogous to the Chirikov standard map.
	
	\textbf{Acknowledgment.} The author sincerely appreciates Prof. Jinxin Xue for engaging discussions and helpful insights on related topics.
	
	\section{Proof of Theorem \ref{Uniformly geometrically ergodic}}
	\quad The primary goal of this section is to prove Theorem \ref{Uniformly geometrically ergodic}, which establishes uniform geometric ergodicity for the one-point chain. Some of the concepts used in the sequel are best introduced in a general abstract setting. We begin by covering some preliminaries on Markov chains in this abstract setting in the following subsection.
	\subsection{Markov chain preliminaries}  
	\quad The projective chain is the Markov chain $\{\check{\Phi}_{\underline{\tau}}^m\}$ on the projective bundle $P\mathcal{O}$ of $\mathcal{O}$ with transition kernel $\check{P}((x,u),\check{A})=\mathbb{P}(\check{\Phi}_{\underline{\tau}}(x,u)\in\check{A})$, its corresponding vector fields are $\check{\mathcal{X}}=\{\check{X_1}, \check{X_2}\}$, where for $(x,u)\in P\mathcal{O}$, by a slight abuse of notation, we use $u$ to denote both an element of the tangent space $T_x\mathcal{O}$ and its equivalence class in $P_x\mathcal{O}$ whenever $u\neq 0$,
	\begin{equation*}
		\check{\Phi}_{\underline{\tau}}^m(x,u)=\left(\Phi_{\underline{\tau}}^m(x), \frac{D_x\Phi_{\underline{\tau}}^m u}{|D_x\Phi_{\underline{\tau}}^m u|}\right),
	\end{equation*}
	and
	\[ 
	\check{X_i}(x,u) = \begin{pmatrix}
		X_i(x) \\
		DX_i(x)u-\langle DX_i(x)u, u \rangle u
	\end{pmatrix},\quad 
	i=1, 2.
	\]
	Thoughout this paper, consider the sequence $\{\psi_{\underline{\tau}}^m\}$, where each term is defined recursively by $\psi_{\underline{\tau}}^m = \psi_{\tau_{2m}}^{(2)} \circ \psi_{\tau_{2m-1}}^{(1)}(\psi_{\underline{\tau}}^{m-1})$. This sequence represents one of the chains above (one point chain, two point chain, or projective chain) on $M$, which is $d$-dimensional with transition kernel $Q$ and vector fields $\mathcal{F}:=\{Y_1,Y_2\}$. We say that $Q$ is \textbf{Feller} if $Qg$ is continuous for any bounded, continuous function $g$ on $M$. We say that a set $A\subset M$ is \textbf{small} for $\{\psi_{\underline{\tau}}^m\}$  if there exists an $m_0>0$, and a positive measure $\nu$ on $M$ such that for any $x\in A$,
	\[
	Q^{m_0}(x,B)\geq \nu(B), 
	\]
	for all $B\in\mathscr{B}(M)$.  We say that $\{\psi_{\underline{\tau}}^m\}$ is \textbf{topologically irreducible} if for every $x\in M$ and open set $U\subset M$, there exists $m_0$ such that 
	\[
	Q^{m_0}(x,U)>0.
	\]
	We say that $\{\psi_{\underline{\tau}}^m\}$ is \textbf{strongly aperiodic} if there exists $x_0\in M$ such that for all open sets $U$ containing $x_0$, we have
	\[
	Q(x_0,U)>0.
	\]
	We say that a function $V:M \to[1,\infty)$ satisfies a \textbf{Lyapunov-Foster drift condition} for $\{\psi_{\underline{\tau}}^m\}$ if there exist $\alpha\in (0,1)$, $b>0$ and a compact set $C\subset M$ such that 
	\[
	QV\leq \alpha V+b\chi_C,
	\]
	where $\chi_C$ is the characteristic function on $C$.
	\begin{thm}[\cite{BCZG}, Theorem 2.3]\label{conditions for uniformly geometrically ergodic}
		Let $\{\psi_{\underline{\tau}}^m\}_{m\geq 0}$ and $Q$ be as above. If $Q$ is a Feller transition kernel and assume the following:
		\begin{itemize}
			\item[(a)] There exists an open small set for $\{\psi_{\underline{\tau}}^m\}$.
			\item[(b)] $\{\psi_{\underline{\tau}}^m\}$ is topologically irreducible.
			\item[(c)] $\{\psi_{\underline{\tau}}^m\}$ is strongly aperiodic.
			\item[(d)] There exists a function $V: X\to[1,\infty)$ that satisfies a Lyapunov-Foster drift condition for $\{\psi_{\underline{\tau}}^m\}$.
		\end{itemize}
		Then $Q$ is $V$-uniformly geometrically ergodic.
	\end{thm}
    \subsection{Uniform geometric ergodicity of the one-point Markov chain}
    \quad The proof of Theorem \ref{Uniformly geometrically ergodic} using the following results. The first is a version of Theorem 1 in \cite[Chapter 3]{J}, it says that if the Lie bracket condition holds at $x\in M$, then due to the surjectivity of $D_t\psi^m(x,t)$, the random dynamics can locally reach any infinitesimal direction in the tangent space from $x$ in arbitrarily small times. 
    
  	\begin{lem}\label{surjective}
    	Let $\{\psi_{\underline{\tau}}^m\}$ be as above. Assume $\text{Lie}_x(\mathcal{F})=T_xM$ for some $x\in M$. Then for any neighborhood $U$ of $x$ and any $T'>0$ there exists a point $y$ in $U$, an $m$ and a $\tilde{t}=(\tilde{t}_1,\cdots,\tilde{t}_{2m})\in\mathbb{R}_+^{2m}:=(0,\infty)^{2m}$ such that $\sum_{i=1}^{2m}\tilde{t}_i\leq T'$ and $\psi^m(x,\cdot): t\mapsto \psi^m(x,t)$ is a submersion at $t=\tilde{t}$ and $\psi^m(x,\tilde{t})=y$, i.e., $D_t\psi^m(x,\tilde{t}):T_{\tilde{t}}\mathbb{R}_+^{2m}\to T_yM$ is surjective.
    \end{lem}
    \begin{proof}
    	Fix any neighborhood $U$ of $x$ and any $T'>0$. Choose $Y_{i_1} \in\mathcal{F}$ such that $Y_{i_1}(x)\neq 0$. Let $\varepsilon_1>0$ be such that $Y_{i_1}(\psi_t^{(i_1)}(x))\neq 0$ for all $t\in(0,\varepsilon_1)$. For sufficiently small $\varepsilon_1<T$, $M_1=\{\psi_t^{(i_1)}(x): t\in(0,\varepsilon_1)\}\subset U$. By the constant rank theorem, $M_1$ is a one-dimensional submanifold of $M$, with $t\mapsto \psi_t^{(i_1)}(x)$ its coordinate map. If $\dim M=\dim M_1$, the proof is finished. Otherwise, since a classical result asserts that if $X$ abd $Y$ are vector fields that are tangent to a submanifold, then their Lie bracket $[X,Y]$ and any linear combination $\alpha X+\beta Y$ are also tangent to this submanifold, there exists a vector field $Y_{i_2}$ in $\mathcal{F}$ such that $Y_{i_2}$ is not tangent to $M_1$.
    	
    	Let $x_1$ be a point in $M_1$ such that $Y_{i_2}(x_1)$ is not colinear with $Y_{i_1}(x_1)$. Denote by $ \hat{t}_{1}$ the point in $(0,\varepsilon_1)$ such that $\psi_{\hat{t}_{1}}^{(i_1)}(x)=x_1$. Then $Y_{i_1}$ and $Y_{i_2}$ are not collinear for all points in a neighborhood $x_1$. Therefore there exist a neighborhood $I_1\in (0,\varepsilon_1)$ of $\hat{t}_1$ and $\varepsilon_2>0$ such that the mapping $F_2$ given by $\psi_{t_2}^{(i_2)}\psi_{t_1}^{(i_1)}(x)$ from $I_1\times (0,\varepsilon_2)$ into $M$ satisfies $\varepsilon_1+\varepsilon_2<T'$, $F_2(I_1\times (0,\varepsilon_2))\subset U$ and the tangent map of $F$ has rank $2$ at each point of $I_1\times (0,\varepsilon_2)$. 
    	
    	Continuing in this way, we obtain a sequence $Y_{i_1}, Y_{i_2},\cdots, Y_{i_m}$ of vector fields in $\mathcal{F}$ with the property that the tangent map of the mapping $F_m(t_1,\cdots,t_m)=\psi_{t_m}^{(i_m)}\circ\cdots\circ\psi_{t_1}^{(i_1)}(x)$ has rank $m$ at a point $\hat{t}=(\hat{t}_{1},\cdots,\hat{t}_{m})\in\mathbb{R}_+^{m}$, which further satisfies $\hat{t}_{1}+\cdots+\hat{t}_{m}<T'$ and $F_m(\hat{t}_{1},\cdots,\hat{t}_{m})\in U$. Again by the constant rank theorem, the image of a neighborhood of $(\hat{t}_{1},\cdots,\hat{t}_{m})$ under the mapping $F_m$ is an $m$-dimensional submanifold $M_m$ of $M\cap U$.
    	
    	The procedure stops precisely when each element of $\mathcal{F}$ is tangent to $M_m$. So there exists a point $\hat{y}=\psi_{\hat{t}_{m}}^{(i_{m})}\circ\cdots\circ\psi_{\hat{t}_{1}}^{(i_1)}(x)$ in $U$, an $m$ and $\hat{t}=(\hat{t}_{1},\cdots,\hat{t}_{m})\in\mathbb{R}_+^{m}$ such that $\hat{t}_{1}+\cdots+\hat{t}_{m}<T'$ and $t\mapsto\psi_{t_{m}}^{(i_{m})}\circ\cdots\circ\psi_{t_{1}}^{(i_1)}(x)$ is a submersion at $\hat{t}$. Then we can take a neighborhood $U_{\hat{y}}\subset U$ of $\hat{y}$ and $\tilde{t}=(\tilde{t}_1,\cdots,\tilde{t}_{2m})\in\mathbb{R}_+^{2m}$ with $\tilde{t}_{2(k-1)+i_k}=\hat{t}_k$ for $k=1,\cdots,m$ and $\tilde{t}_{2(k-1)+j}>0$ for $k=1,\cdots,m$ and $\{1,2\}\ni j\neq i_k$ sufficiently small such that $\sum_{i=1}^{2m}\tilde{t}_i\leq T'$, $t\mapsto\psi^m(x,t)$ is a submersion at $\tilde{t}$ and $y:=\psi^m(x,\tilde{t})\in U_{\hat{y}}\subset U$.
    \end{proof}
 
    Secondly, we present the subsequent lemma, which essentially provides a lower bound on the probability that the image of the random variable lies in a given set, under certain geometric and probabilistic conditions.
    \begin{lem}\label{absolutely continuous}
    	Let $\tau$ be a continuous random variable in $U$ which is an open subset of $\mathbb{R}^m$, with density $\rho$. Let $d\leq m$, and let $\phi:U\times M\to M$, $(t,x)\mapsto \phi_x(t)$ be a $C^1$ map. If for some $(t_0,x_0)\in U\times M$ the map $\phi_{x_0}$ is a submersion at $t_0$ and $\rho$ is bounded below by $c_0>0$ on a neighborhood of $t_0$. Then there exist a constant $c>0$ and neighborhood $U_{x_0}$ of $x_0$ and $U_{\hat{x}}$ of $\hat{x}:=\phi(t_0,x_0)$ such that 
    	\begin{equation}\label{lower bound}
    		\mathbb{P}(\phi(\tau,x)\in A)\geq c{\text{Leb}}_M(A\cap U_{\hat{x}}), \ \ \text{for all}\ \  x\in U_{x_0},\ A\in\mathscr{B}(M).
    	\end{equation}
    \end{lem}
    \begin{proof}
    	Write $t=(t^1,t^2)\in\mathbb{R}^d\times\mathbb{R}^{m-d}$. Since $\phi_{x_0}$ is a submersion at $t_0$, by the constant rank theorem, we may assume $\frac{\partial \phi_x(t^1,t^2)}{\partial t^1}$ is surjective on some neighborhood $U_0$ of $x_0$ and $V_0$ of $t_0$. Define 
    	\begin{align*}
    		\tilde{\phi}_x: V_0&\to M\times \mathbb{R}^{m-d},\\
    		(t^1,t^2) &\mapsto (\phi_x(t^1,t^2),t^2).
    	\end{align*}
    	Then $D_t\tilde{\phi}_x$ is invertible on $V_0$ for any $x\in U_0$. Note that $\tilde{\phi}_{x_0}(t_0)=(\hat{x},t_0^2)$, we can choose a neighborhood $U_{\hat{x}}$ of $\hat{x}$, $U_2$ of $t^2$ and $U_{x_0}$ of $x_0$ such that $|\det D\tilde{\phi}_x^{-1}(y)|\geq c_1>0$ for any $x\in U_{x_0}$ and any $y\in U_{\hat{x}}\times U_2$.
    	
    	Write the random variable $\tau$ as a couple $(\tau^1,\tau^2)$, and let $B\subset U_{x_0}$ be a Borel set,
    	\begin{align*}
    		\mathbb{P}(\phi(x,\tau)\in B)&\geq \mathbb{P}(\phi(x,\tau)\in B, \tau^2\in U_2)\\
    		&\geq \mathbb{P}(\tilde{\phi}_x(\tau)\in B\times U_2)\\
    		&=\mathbb{P}(\tau\in \tilde{\phi}_x^{-1}(B\times U_2))\\
    		&=\int_{\tilde{\phi}_x^{-1}(B\times U_2)}\rho(t)dt\geq c_0\int_{\tilde{\phi}_x^{-1}(B\times U_2)}dt\\
    		&=c_0\int_{B\times U_2}|\det D\tilde{\phi}_x^{-1}(y)|dy\\
    		&\geq c_0c_1\int_{B\times U_2}dy=c_0c_1\text{Leb}_M(B)\text{Leb}_{\mathbb{R}^{m-d}}(U_2).
    	\end{align*}
    	Therefore $(\ref{lower bound})$ holds with $c=c_0c_1\text{Leb}_{\mathbb{R}^{m-d}}(U_2)$.
    \end{proof}
    We are now prepared to prove the existence of an open small set for $\{\psi_{\underline{\tau}}^m\}$.
  \begin{lem}\label{small set}
  	Let $\{\psi_{\underline{\tau}}^m\}$ and $Q$ be as above. Assume $\text{Lie}_x(\mathcal{F})=T_xM$ for some $x\in M$. Then there exists an open small set for $\{\psi_{\underline{\tau}}^m\}$.
  \end{lem}
  \begin{proof}
  	It follows from $\text{Lie}_x(\mathcal{F})=T_xM$ for some $x\in M$ and Lemma \ref{surjective} that there exists $\tilde{t}=(\tilde{t}_1,\cdots,\tilde{t}_{2m})\in[0,T]^{2m}$ such that $\psi^m(x,\cdot)$ is a submersion at $\tilde{t}$. Then by Lemma \ref{absolutely continuous}, there exist a constant $c>0$ and neighborhood $U$ of $x$ and $U'$ of $\psi^m(x,\tilde{t})$ such that
  	\begin{equation*}
  		\mathbb{P}(\psi_{\underline{\tau}}^m(x)\in A)\geq c\text{Leb}_M(A\cap U'), \ \ \text{for all}\ \  x\in U,\ A\in\mathscr{B}(M).
  	\end{equation*}
  	We conclude $U$ is a $\nu$-small set with $\nu(A):=c\text{Leb}_M(A\cap U')$.
  \end{proof}
  Next, we prove that the one-point Markov chain is topologically irreducible. We say that a set $A\subseteq M$ is \textbf{reachable} from $B\subseteq M$ for $\{\psi_{\underline{\tau}}^m\}$ if for every point $x\in B$, there exists $m\geq 0$ and $\underline{\tau}=(\tau_1,\tau_2,\cdots,\tau_{2m})\in [0,\infty)^{2m}$ such that $\psi_{\underline{\tau}}^m(x)\in A$. We say that $\{\psi_{\underline{\tau}}^m\}$ is \textbf{exactly controllable} if for any $x, y\in M$, there exists $m\geq 0$ and $\underline{\tau}=(\tau_1,\tau_2,\cdots,\tau_{2m})\in [0,\infty)^{2m}$ such that $\psi_{\underline{\tau}}^m(x)=y$. Clearly, exactly controllability of $\{\psi_{\underline{\tau}}^m\}$ implies that it is topologically irreducible. Finally, we prove that for any $(q,p)\in \mathcal{O}$, there exists $m\geq 0$ and $\underline{\tau}=(\tau_1,\tau_2,\cdots,\tau_{2m})\in [0, T]^{2m}$ such that $\Phi_{\underline{\tau}}^m(q,p)=(q_0,p_0)$. In fact, $\psi_{\underline{\tau}}$ corresponding to the sequence $(\tau_1,\tau_2)$ is the same as $\psi_{\underline{\tau}}^2$ corresponding to the sequence $(\tau_1/2,0,\tau_1/2,\tau_2)$, so by making finitely many such substitutions we can ensure that $\tau_i\leq T$ for all $i$.
  
  \begin{lem}\label{irreducibility}
  	Let $\{{\Phi}_{\underline{\tau}}^m\}$ be as above. Then, $\{{\Phi}_{\underline{\tau}}^m\}$ is exactly controllable.
  \end{lem}
  \begin{proof}
  	Since $f_i$, $i=1, 2$ are analytic on $\mathbb{S}^1$, let $p_0$ be the point in $\mathbb{S}^1$ where $|f_1|$ attains its maximum value, and $q_0$ be the point in $\mathbb{S}^1$ where $|f_2|$ attains its maximum value. It is evident that  $(q_0, p_0)\in\mathcal{O}$ and $f_1'(p_0)=0$, $f_2'(q_0)=0$. 
  	
  	Next, we prove that for any $(q,p)\in \mathcal{O}$, there exists $m\geq 0$ and $\underline{\tau}=(\tau_1,\tau_2,\cdots,\tau_{2m})\in [0,\infty )^{2m}$ such that $\Phi_{\underline{\tau}}^{m}(q,p)=(q_0,p_0)$. Since $(q,p)\in \mathcal{O}$, then either $f_1(p)\neq 0$ or $f_2(q)\neq 0$. Without loss of generality we may assume $f_1(p)\neq 0$. In fact, if $f_1(p)= 0$, then $f_2(q)\neq 0$, set $\tau_1=0$, and almost every $\tau_2>0$ such that $f_1(p+\tau_2f_2(q))\neq 0$. Choose $k, j\in \mathbb{Z}$ such that $\tau_1=\frac{q_0+2k\pi-q}{f_1(p)}\geq 0$ and $\tau_2=\frac{p_0+2j\pi-p}{f_2(q_0)}\geq 0$, then $\varphi_{\tau_2}^{(2)}\circ\varphi_{\tau_1}^{(1)}(q,p)=(q_0,p_0)$. We conclude by noting that for any two points $(q,p)$ and $(q_1,p_1)$ there exists $\tilde{m}\geq 0$ and $\underline{\tilde{\tau}}=(\tilde{\tau}_1,\tilde{\tau}_2,\cdots,\tilde{\tau}_{2\tilde{m}})$ such that $\Phi_{\underline{\tilde{\tau}}}^{\tilde{m}}(q,p)=(q_1,p_1)$ by first traveling to $(q_0,p_0)$ and then traveling to $(q_1,p_1)$ (using the same arguments as above, but in reverse).
  \end{proof}
  Now we proceed to prove the Lyapunov-Foster drift condition for $\{\Phi_{\underline{\tau}}^m\}$, which is inspired by \cite[Lemma 13]{Co}.
  \begin{pro}\label{Lyapunov-Foster drift condition}
  	Let $\{\Phi_{\underline{\tau}}^m\}$ be as above. Assume $(H1)$ holds. Then there exists a continuous function $V:\mathcal{O}\to[1,\infty)$ which satisfies the Lyapunov-Foster drift condition for $\{\Phi_{\underline{\tau}}^m\}$ .
  \end{pro}	
  \begin{proof}
  	Let $K>1$ such that 
  	\begin{equation*}
  		\frac{1}{K}d(z, C_{f_i})\leq|f_i(z)|\leq K d(z, C_{f_i}), \ i=1, 2
  	\end{equation*}
    for any $z\in\mathbb{S}^1$, such a $K$ exists by $(H1)$. 
  	
  	For any $(q,p)\in \mathcal{O}$, denote
  	\begin{equation*}
  		V(q,p)=\max(d(q,C_{f_2}),d(p,C_{f_1}))^{-\beta}+b
  	\end{equation*}
  	where $\beta>0$ will be chosen to be small and $b>0$ is a constant chosen so that $V(q,p)\geq 1$ for all $(q,p)\in \mathcal{O}$.
  	
  	Fix any $q_0\in C_{f_2}$ and $p_0\in C_{f_1}$. In order to prove that $V$ satisfies the Lyapunov-Foster drift condition for $\{\Phi_{\underline{\tau}}^m\}$, we only need to prove that there exist $\alpha<1$ and $\varepsilon_0>0$ such that 
  	\begin{equation*}
  		PV(q,p)\leq\alpha V(q,p),\quad (q,p)\in B_{\varepsilon_0}(q_0,p_0)\cap \mathcal{O},
  	\end{equation*}
  	where $B_{\varepsilon_0}(q_0,p_0)=\{(q,p)\in \mathbb{T}^2: d(q, q_0)+d(p, p_0)<\varepsilon_0\}$.
  	
  	Fix any $(q,p)\in B_{\varepsilon_0}(q_0,p_0)\cap \mathcal{O}$.
  
  	\begin{itemize}
  		\item[(1)] We first consider the case that $d(q, q_0)\leq d(p, p_0)$. Choose $\varepsilon_0>0$ small enough such that 
  		\[
  		TK\varepsilon_0<\frac{r}{8\pi},
  		\]
  		where
  		\[
  		r:=\min\{d(q_0, C_{f_2}\setminus \{q_0\}), d(p_0, C_{f_1}\setminus \{p_0\})\}.
  		\] 
  		Now we split into three sub-cases:
  		\begin{itemize}
  			\item[(a)] Denote 
  			\[
  			A=\left\{(\tau_1,\tau_2)\in [0,T]^2: d(q+\tau_1 f_1(p), q_0)\leq \frac{d(p, p_0)}{2KT}\right\}.
  			\]
  			It follows from $|f_1(p)|\geq \frac{1}{K}d(p, p_0)$ that 
  			\[
  			\mathbb{P}(A)\leq \frac{1}{T^2}.
  			\]
  			And since $|f_2(q')|\leq Kd(q',q_0)$ for any $|q'-q_0|<\frac{r}{6\pi}$, we see that 
  			\[
  			d(p+\tau_2f_2(q+\tau_1f_1(p)), p_0)\geq \frac{d(p, p_0)}{2},\text{ for all } (\tau_1,\tau_2)\in A
  			\]
  			Thus $V(\Phi_{\underline{\tau}}(q,p))\leq 2^\beta V(q,p)$ for any $(\tau_1,\tau_2)\in A$.
  			\item[(b)] Denote
  			\[
  			B=\left\{(\tau_1,\tau_2)\in [0,T]^2: \frac{d(p, p_0)}{2KT}<d(q+\tau_1 f_1(p), q_0)\leq 2d(p,p_0)\right\}
  			\]
  			Then 
  			\[
  			\mathbb{P}(B)\leq \frac{4K}{T}.
  			\]
  			And $V(\Phi_{\underline{\tau}}(q,p))\leq (2KT)^\beta V(q,p)$ for any $(\tau_1,\tau_2)\in B$.
  			\item[(c)] In the complement of $A\cup B$, we see that $V(\Phi_{\underline{\tau}}(q,p))\leq 2^{-\beta} V(q,p)$.
  		\end{itemize}
  		Based on the above, we conclude:
  		\begin{align}\label{22}
  			PV(q,p):=&\mathbb{E}(V(\Phi_{\underline{\tau}}(q,p)))\nonumber\\
  			\leq&\left(\frac{K 2^\beta}{T^2}+\frac{2^{2+\beta} K^{\beta+1}}{T^{1-\beta}}+2^{-\beta}\right)V(q,p).
  		\end{align}
  		\item[(2)] If $d(q, q_0)> d(p, p_0)$, then we split into four sub-cases:
  		\begin{itemize}
  			\item[(a)] As before, denote
  			\[
  			A=\left\{(\tau_1,\tau_2)\in [0,T]^2: d(q+\tau_1 f_1(p), q_0)\leq \frac{d(p, p_0)}{2KT}\right\}
  			\]
  			Then $\mathbb{P}(A)\leq \frac{1}{T^2}$ and $V(\Phi_{\underline{\tau}}(q,p))\leq 2^\beta d(p,p_0)^{-\beta}$. We only need to consider the case that $\mathbb{P}(A)>0$, this requires $d(q,q_0)-KTd(p,p_0)\leq d(p, p_0)/(2KT)$, and hence $d(p,p_0)\geq d(q,q_0)/(2KT)$ and thus $V(\Phi_{\underline{\tau}}(q,p))\leq (4KT)^\beta V(q,p)$ for any $(\tau_1,\tau_2)\in A$.
  			\item[(b)] Unlike before, we define 
  			\begin{equation*}
  				B=\left\{(\tau_1,\tau_2)\in [0,T]^2: \frac{d(p, p_0)}{2KT}<d(q+\tau_1 f_1(p), q_0)\leq \frac{d(q, q_0)}{\sqrt{T}}\right\}.
  			\end{equation*}
  			Then 
  			\[
  			\mathbb{P}(B)\leq \frac{2Kd(q,q_0)}{T^{3/2}d(p,p_0)}.
  			\]
  			We only need to consider the case that $\mathbb{P}(B)>0$, this requires
  			\[
  			d(q,q_0)-KTd(p,p_0)\leq \frac{d(q, q_0)}{\sqrt{T}},
  			\]
  			and hence $d(p, p_0)\geq \frac{d(q,q_0)}{2KT}$.
  			So it follows that $\mathbb{P}(B)\leq \frac{4K^2}{\sqrt{T}}$ and $V(\Phi_{\underline{\tau}}(q,p))\leq (4K^2T^2)^\beta V(q,p)$ for any $(\tau_1,\tau_2)\in B$.
  			\item[(c)] Denote
  			\[
  			C=\left\{(\tau_1,\tau_2): d(q+\tau_1 f_1(p), q_0)> \frac{d(q, q_0)}{\sqrt{T}}, |\tau_2f_2(q+\tau_1 f_1(p))|\leq 3d(q,q_0)\right\}.
  			\]
  			Then $\mathbb{P}(C)\leq \frac{3K}{\sqrt{T}}$. And $V(\Phi_{\underline{\tau}}(q,p))\leq T^{\beta/2}V(q,p)$ for any $(\tau_1,\tau_2)\in C$.
  			\item[(d)] In the complement of $A\cup B\cup C$, it follows from $d(p+\tau_2f_2(q+\tau_1 f_1(p)),p_0)\geq 2d(q,q_0)$ that $V(\Phi_{\underline{\tau}}(q,p))\leq 2^{-\beta} V(q,p)$.
  		\end{itemize}
  		Based on the above, we conclude:
  		\begin{align}\label{23}
  			PV(q,p):=&\mathbb{E}(V(\Phi_{\underline{\tau}}(q,p)))\nonumber\\
  			\leq&\left(\frac{4^\beta K^\beta}{T^{2-\beta}}+\frac{4^{1+\beta}K^{2+2\beta}}{T^{1/2-2\beta}}+\frac{3K}{T^{(1-\beta)/2}}+2^{-\beta}\right)V(q,p).
  		\end{align}
  	\end{itemize}
    Using equations \ref{22} and \ref{23}, by choosing $0<\beta<1/2$ sufficiently small and $T$ sufficiently large (depending on $\beta$ and $K$), we obtain the desired conclusion.
  \end{proof}

  \begin{proof}[\textbf{Proof of Theorem \ref{Uniformly geometrically ergodic}}]
  	  Observe that $\Phi_{\underline{\tau}}$ is a volume-preserving diffeomorphism of $\mathcal{O}$ for every $\underline{\tau}$. It is evident that the Lebesgue probability measure $\mu$ on $\mathcal{O}$ is a stationary measure for the Markov chain $\{\Phi_{\underline{\tau}}^m\}_{m=0}^{\infty}$. The Lie brackets of $X_1$ and $X_2$ are given by 
  	  \begin{equation*}
  		  [X_1,X_2](x)=DX_2(x)X_1(x)-DX_1(x)X_2(x)=
  		  \begin{pmatrix}
  			  -f_1'(p)f_2(q)\\
  			  f_1(p)f_2'(q)
  		  \end{pmatrix}
  	  \end{equation*}
  	  and the following matrix
  	  \begin{align*}
  		  (X_1(x),X_2(x), [X_1,X_2](x))=
  		  \begin{pmatrix}
  			  f_1(p) & 0 & -f_1'(p)f_2(q)\\
  			  0 & f_2(q) & f_1(p)f_2'(q)
  		  \end{pmatrix}
  	  \end{align*}
  	  has rank 2 for all $x=(q,p)\in\mathcal{O}$. Thus $\text{Lie}_x(\mathcal{X})=T_x\mathcal{O}$ for all $x\in \mathcal{O}$. Strong aperiodicity follows immediately from $\mathbb{P}(\tau<\varepsilon)>0$ for any $\varepsilon>0$ and $\Phi_0(x)=x$. By combining Lemma \ref{small set}, Lemma \ref{irreducibility}, and Proposition \ref{Lyapunov-Foster drift condition}, we establish that the conditions of Theorem \ref{conditions for uniformly geometrically ergodic} are satisfied. Consequently, the transition kernel $P$ of $\{\Phi_{\underline{\tau}}^m\}$ is $V$-uniformly geometrically ergodic, with $\mu$ as the unique stationary measure.
  \end{proof}
 
  \section{Proof of Theorem \ref{Positivity of the top Lyapunov exponent}}
\quad Theorem \ref{Positivity of the top Lyapunov exponent} makes two assertions: firstly, the existence and almost sure constancy of the Lyapunov exponent $\lambda_1$, and secondly, that $\lambda_1$ is strictly positive. We use $|u|$ to represent the norm of a vector $u$ in the tangent space $T_x\mathcal{O}$, and $\|A\|$ for the operator norm of a linear mapping $A: T_x\mathcal{O} \rightarrow T_y\mathcal{O}$. For any positive real number $r$, we define $\log^+(r) = \max(\log r, 0)$. Let $\mu$ be as in Theorem \ref{Uniformly geometrically ergodic}, which is the unique stationary measure on $\mathcal{O}$. Let us first outline how to prove the first one using tools from random dynamical systems theory. Given that $f_i$ belongs to $C^{\omega}(\mathbb{S}^1, \mathbb{R})$ and the set $\mathcal{O}\subset\mathbb{T}^{2}$ is bounded, we can establish the following integrability condition:
\begin{equation*}
	\mathbb{E}\int_{\mathcal{O}} (\log^+\|D_x\Phi_{\underline{\tau}}(x)\|+\log^+\|D_x\Phi_{\underline{\tau}}(x)^{-1}\|)\mu(dx)<\infty.
\end{equation*}
This condition being met, the multiplicative ergodic theorem subsequently assures the existence of the Lyapunov exponent and is constant for $\mu$-a.e. $x\in\mathcal{O}$ and almost every $\underline{\tau}$. Since $\Phi_{\underline{\tau}}^m$ is volume-preserving, $\det(D_x\Phi_{\underline{\tau}}^m)$ is identically equal to 1. Consequently, the sum of the Lyapunov exponents, $\lambda_{\Sigma}$, defined as the limit of $\frac{1}{m}\log |\det(D_x\Phi_{\underline{\tau}}^m)|$ as $m$ approaches infinity, equals zero, i.e.,  $\lambda_{\Sigma}=\lim_{m\to\infty}\frac{1}{m}\log |\det(D_x\Phi_{\underline{\tau}}^m)|=0$. The task ahead involves proving that the Lyapunov exponent, $\lambda_1$, does not degenerate to zero, ensuring its positive value. Let's proceed to discuss the consequences that arise from the condition $2\lambda_1= \lambda_{\Sigma}$. 

\subsection{Consequences of \boldmath $2\lambda_1=\lambda_{\Sigma}$}
\quad For each positive integer $m$, we define the pushforward of $\mu$ by $\Phi_{\underline{\tau}}^m$ as the probability measure $\mu_m$ on $\mathcal{O}$, which is given by
\begin{equation*}
	\mu_m(A):=(\Phi_{\underline{\tau}}^m)_*\mu(A):=\mu((\Phi_{\underline{\tau}}^m)^{-1}(A)).
\end{equation*}  
This construction allows us to track how the measure $\mu$ evolves under the action of $\Phi_{\underline{\tau}}^m$. For probability measure $\mu$ and $\mu_m$ on $\mathcal{O}$, following Donsker and Varadhan \cite[Lemma 2.1]{DV}, define the entropy of $\mu_m$ with respect to $\mu$ by 
\[h(\mu;\mu_m)=
\begin{cases}
	\int\frac{d\mu_m}{d\mu}(x)\log(\frac{d\mu_m}{d\mu}(x))\mu(dx) & \text{if} \ \ \mu_m\ll \mu,\\
	\infty & \text{otherwise},
\end{cases}
\]
where $\mu_m\ll\mu$ means that $\mu_m$ is absolutely continuous with respect to $\mu$ and $d\mu_m/d\mu$ is the Radon-Nikodym derivative. The measure $\mu_m$, which varies depending on $\underline{\tau}$ and is characterized as a random measure, leads us to consider its average relative entropy with $\mu$, symbolized as $\mathbb{E}h(\mu;\mu_m)$. The finiteness of this value is crucial for our subsequent analysis. Specifically, in cases where $\mu_m$ equates to $\mu$ almost surely, as in our situation, it follows that $\mathbb{E}h(\mu;\mu_m)=\mathbb{E}h(\mu;\mu)=0$.

\begin{thm}\label{degeneracy}
	Let $\{\Phi_{\underline{\tau}}^m\}$ be as above. Assume $(H1)$ holds and $2\lambda_1=\lambda_\Sigma$. Then either
	\begin{itemize}
		\item[(a)] there exists a Riemannian structure $\{g_x:x\in \text{Supp}(\mu)\}$ on $\text{Supp}(\mu)$ such that 
		\begin{equation}\label{conformal}
			g_{\Phi_t^m(x)}(D_x\Phi_t^m(x)u,D_x\Phi_t^m(x)v)=\det D_x\Phi_t^m(x) g_x(u,v)
		\end{equation}
		for all $m$, $t\in[0,T]^{2m}$, $x\in \text{Supp}(\mu)$ and $u,v\in T_x\mathcal{O}$; or
		\item[(b)] there exist proper linear subspaces $E_x^1,\cdots,E_x^p$ of $T_x \mathcal{O}$ for all $x\in \text{Supp}(\mu)$ such that 
		$$D_x\Phi_t^m(E_x^i)=E_{\Phi_t^m(x)}^{\sigma(i)}, \ \ 1\leq i\leq p$$
		for all $m$, $t\in[0,T]^{2m}$, $x\in \text{Supp}(\mu)$ and some permutation $\sigma$.
	\end{itemize}
\end{thm}
Theorem \ref{degeneracy} mainly combines results from \cite{B}, leading up to Theorem 6.8 in \cite{B}. Before proving this theorem, we first establish some preliminary notations. Recall the projective chain
\begin{equation*}
	\check{\Phi}_{\underline{\tau}}^m(x,v):=\left(\Phi_{\underline{\tau}}^m(x),\frac{D_x\Phi_{\underline{\tau}}^m(x)v}{|D_x\Phi_{\underline{\tau}}^m(x)v|}\right)
\end{equation*}
on the projective bundle $P\mathcal{O}$ of $\mathcal{O}$, where, by a slight abuse of notation, we use $v$ to denote both an element of the tangent space $T_x\mathcal{O}$ and its equivalence class in $P_x\mathcal{O}$ whenever $v\neq 0$. And denote 
\[
A_{\underline{\tau},x}^m(v):=\frac{D_x\Phi_{\underline{\tau}}^m(x)v}{|D_x\Phi_{\underline{\tau}}^m(x)v|}.
\]
Let $\mathscr{M}(P\mathcal{O})$ denote the space of Borel probability measures on $P\mathcal{O}$, with the weak topology. If $\pi:P\mathcal{O}\to\mathcal{O}$ is the natural bundle map, let $\mathscr{M}_{\mu}(P\mathcal{O})$ denote the set of $\nu\in\mathscr{M}(P\mathcal{O})$ such that $\nu \pi^{-1}=\mu$. Since the bundle has compact fibre, the set $\mathscr{M}_{\mu}(P\mathcal{O})$ is uniformly tight and hence it is a compact subset of $\mathscr{M}(P\mathcal{O})$. For any probability measure $\nu$ on $P\mathcal{O}$ such that $\pi_*\nu=\mu$ (i.e., $\nu\in\mathscr{M}_{\mu}(P\mathcal{O})$), we denote by $\{\nu_x:x\in\mathcal{O}\}$ the regular conditional probability distribution of $\nu$ given $\pi$. The $\nu_x$ are defined for $\mu$-a.e. $x\in\mathcal{O}$ and each $\nu_x$ is a probability measure on $P\mathcal{O}$ such that $\nu_x(P_x\mathcal{O})=1$.

$\nu\in \mathscr{M}_{\mu}(P\mathcal{O})$ is said to be \textbf{invariant conditional distribution} if, for all $m$, $\mu_m\ll \mu$ and $\mu\{x:(A_{\underline{\tau},x}^m)_*\nu_x=\nu_{\Phi_{\underline{\tau}}^m(x)}\}=1$ for almost every $\tau$. By Corollary 5.6 in \cite{B}, assume integrability condition and $\mathbb{E}h(\mu;\mu_m)<\infty$, if $2\lambda_1=\lambda_{\Sigma}$ then there exists $\nu\in\mathscr{M}_{\mu}(P\mathcal{O})$ satisfing invariant conditional distribution. The slight difference between Theorem \ref{degeneracy} and Theorem 6.8 in \cite{B} lies in the fact that, in our case, $V$-uniformly geometrically ergodic of $P$ implies any invariant conditional distribution $\nu$ has a version such that $x\mapsto \nu_x$ is continuous on $\text{Supp}(\mu)$. This permits us to apply the statement and proof of Proposition 6.3 in \cite{B} to our situation, as shown below.

\begin{lem}\label{RCPD}
	 Let $\{\Phi_{\underline{\tau}}^m\}$ be as above. Assume $(H1)$ holds and $2\lambda_1=\lambda_{\Sigma}$, then there exists a version of the regular conditional probability distribution of $\nu$ such that $x\mapsto \nu_x$ is continuous on $\text{Supp}(\mu)$.   
\end{lem}
\begin{proof}
	First, we prove that there exists an increasing sequence of compact subsets $K_1\subset K_2\subset\cdots$ of $\mathcal{O}$ such that for each $n$, we have that $\mu(K_n)\geq 1-1/n$ and $x\mapsto \nu_x$ is continuous along $x\in K_n$. In fact, since $\mathbb{S}^1$ is a compact space, by Stone-Weierstrass Theorem, there exists $\{g_k\}\subset C(\mathbb{S}^1,\mathbb{R})$ is dense in $C(\mathbb{S}^1,\mathbb{R})$. For any $k$, and any $n$, by Lusin's Theorem, there exists compact subset $C_{n,k}$ of $\mathcal{O}$ such that 
	\[
	\mu(C_{n,k})\geq 1-\frac{1}{n2^k},
	\]
	and the measurable function
	\[
	G_k(x):=\int_{\mathbb{S}^1} g_k(v)d\nu_x(v)
	\]
	is continuous on $C_{n,k}$. Now let
	\[
	C_n:=\bigcap_{k=1}^{\infty}C_{n,k},\quad
    K_n:=\bigcup_{m=1}^n C_m.
	\]
	Then 
	\[
	\mu(K_n)\geq \mu(C_n)\geq 1-\frac{1}{n},
	\]
	and for any $k$, $G_k$ is continuous on $K_n$ for each $n$, and since $\{g_k\}\subset C(\mathbb{S}^1,\mathbb{R})$ is dense in $C(\mathbb{S}^1,\mathbb{R})$, so $x\mapsto \nu_x$ is continuous along $x\in K_n$.
	
	Fixed any $g\in C(\mathbb{S}^1,\mathbb{R})$, denote 
	\[
	F_g(x,A)=\int g(Av)d\nu_x(v).
	\]
	Then $F_g|_{K_n\times GL_2(\mathbb{R})}$ is continuous. By Tietze extension theorem, for each $n$ there exists a function $F_{g,n}: \mathcal{O}\times GL_2(\mathbb{R})\to\mathbb{R}$ such that 
	\[
	F_g|_{K_n\times GL_2(\mathbb{R})}=F_{g,n}|_{K_n\times GL_2(\mathbb{R})},
	\]
	and such that $\|F_{g,n}\|_{L^\infty}\leq\|F_g\|_{L^\infty}$. Lastly, define $\chi_n$ be the indicator function of $K_n\times GL_2(\mathbb{R})$ and observe that $\chi_n F_g=\chi_n F_{g,n}$.
	
	Denote 
	\[
	G(x):=\int_{\mathbb{S}^1} g(v) d\nu_x(v),
	\]
	and for $F: \mathcal{O}\times GL_2(\mathbb{R})\to\mathbb{R}$, denote $Q^n F:\mathcal{O}\to\mathbb{R}$ by
	\[
	Q^nF(x):=\mathbb{E}[F(\Phi_{\underline{\tau}}^n(x), (A_{\underline{\tau},x}^n)^{-1})].
	\]
	By Corollary 5.6 in \cite{B}, there exists $\nu\in\mathscr{M}_{\mu}(P\mathcal{O})$ satisfing invariant conditional distribution.	Let $\mathcal{O}'\subset \mathcal{O}$ be the $\mu$-full measure set such that 
	\[
	\nu_x=\mathbb{E}[(A_{\underline{\tau},x}^n)_*^{-1}\nu_{\Phi_{\underline{\tau}}^n(x)}].
	\]
	Then for any $x\in\mathcal{O}'$, we have 
	\begin{align*}
		|G(x)-Q^nF_{g,n}(x)|&=|Q^nF_g(x)-Q^nF_{g,n}(x)|\\
		&\leq |Q^n[\chi_n F_g](x)-Q^n[\chi_n F_{g,n}](x)|+2\|F_g\|_{L^\infty}|Q^n[1-\chi_n](x)|\\
		&=2\|F_g\|_{L^\infty}|Q^n[1-\chi_n](x)|\\
		&=2\|F_g\|_{L^\infty} P^n\chi_{K_n^c}(x)\\
		&\leq 2\|F_g\|_{L^\infty}\left(\frac{1}{n}+CV(x)\gamma^n\right),
	\end{align*}
    where the last inequality holds because $P$ is $V$-uniformly geometrically ergodic with continuous function $V:\mathcal{O}\to [1,\infty)$ by Theorem \ref{Uniformly geometrically ergodic}. Therefore, for any compact subset $K'$, $G|_{\mathcal{O}'\cap K'}$ is the uniform limit of continuous functions $Q^n F_{g,n}|_{\mathcal{O}'}$.  It then follows from the local compactness of $\mathcal{O}$ that $G|_{\mathcal{O}'}$ is continuous. To prove this Lemma, it sufficies to show that for any $g\in C(\mathbb{S}^1,\mathbb{R})$, there is a continuous function $G':\text{Supp}(\mu)\to\mathbb{R}$ such that $G'=G$ holds $\mu$-a.e.. Since $\mathcal{O}'\subset \text{Supp}(\mu)$ is dense, the proof is complete.
\end{proof}

\begin{proof}[\textbf{Proof of Theorem \ref{degeneracy}}]
    Using Theorem 6.8 in \cite{B}, it suffices to show that there exists a version of the regular conditional probability distribution of $\nu$ such that $x\mapsto \nu_x$ is continuous on $\text{Supp}(\mu)$. By the assumption, this result is a direct consequence of Theorem \ref{Uniformly geometrically ergodic} and Lemma \ref{RCPD}.
\end{proof}

\subsection{Strong Feller}
\quad If Assumption $(H1)$ holds and $2\lambda_1=\lambda_{\Sigma}$, then Theorem \ref{degeneracy} $(a)$ or $(b)$ holds on $\text{Supp}(\mu)$.  In this subsection, we provide some auxiliary results to rule out Theorem \ref{degeneracy} $(a)$ or $(b)$. Let $\{\psi_{\underline{\tau}}^m\}$, $Q$, $M$ and $\mathcal{F}$, etc., be as introduced in Section 2. We say that $Q$ is \textbf{Strong Feller} if $Qg$ is continuous for any bounded, measurable function $g$ on $M$.
\begin{thm}\label{Strong Feller}
	Let $\{\psi_{\underline{\tau}}^m\}$ and $Q$ be as above. Assume $\text{Lie}_{x_0}(\mathcal{F})=T_{x_0} M$ for some $x_0\in M$. Then there exist some $m$ and open neighborhood $U$ of $x_0$ such that the transition kernel $Q^{m}$ is strong Feller on $U$.
\end{thm}
In preparation for the proof of this theorem, we will first present a few lemmas. For any $m\in\mathbb{N}$, $x\in M$, consider $\psi_x^m:(0,T)^{2m}\to M$, $\psi_x^m(t):=\psi_t^m(x)$. Let $J{\psi_x^m}(t)$ represent the Jacobian determinant of $\psi_x^m$, then
\begin{equation*}
	|J{\psi_x^m}(t)|=(\det(D\psi_x^m(t)D\psi_x^m(t)^T))^{1/2},
\end{equation*}
where $A^T$ is the transpose of matrix A.
\begin{lem}\label{submersion for almost all time}
	Assume $\text{Lie}_{x_0}(\mathcal{F})=T_{x_0} M$ for some $x_0\in M$, then there exist some $m$ and open neighborhood $U$ of $x_0$ such that for every $x\in U$, the map $t\mapsto \psi^{m}_{x}(t)$ is a submersion for almost every $t\in(0,T)^{2m}$. Moreover, for any bounded, measurable function $f$ on $M$, writting $t=(t_{2m},\cdots,t_1)$ and $dt=dt_{2m}\cdots dt_1$,
	\begin{equation*}
		Q^mf(x)=\int_{M}f(y)\left(\int_{\{t\in(0,T)^{2m}:\psi_{t}^{m}(x)=y\}}\frac{1}{T^{2m}|J{\psi_x^m}(t)|}\mathcal{H}^{2m-d}(dt)\right)\text{Leb}_M(dy),\ \ \text{for any} \ x\in U,
	\end{equation*}
	where $\mathcal{H}^{2m-d}$ is $(2m-d)$-dimensional Hausdorff measure on $(0,T)^{2m}$.
\end{lem}
\begin{proof}
	It follows from Theorem \ref{surjective} that there exist $m$ and $t\in (0,T)^{2m}$ such that $\psi_{x_0}^{m}: (0,T)^{2m}\to M$ is a submersion at $t$. Define $g:M\times (0,T)^{2m}\to\mathbb{R}$ by 
	\begin{equation*}
		g(x,t):=|J{\psi_x^{m}}(t)|.
	\end{equation*}
	Then $g(x_0,t)>0$. With the function $g_{t}(\cdot) = g(\cdot, t): M\to \mathbb{R}$ being continuous, it naturally follows that the set $U= g_{t}^{-1}((0, \infty))$ defines an open neighborhood around $x_0$ in $M$. Now since the vector fields (and hence their flows) are analytic and analyticity is preserved under addition, multiplication, composition and differentiation, $g_{x}(\cdot):=g(x,\cdot):(0,T)^{2m}\to\mathbb{R}$ is analytic for all $x\in M$. Consequently, based on a result that is widely acknowledged, with instances provided in \cite{M} and \cite[Lemma 5.22]{KP}, it follows that the set $\mathcal{N}(g_{x}):=\{t\in(0,T)^{2m}:g_{x}(t)=0\}$ is a zero Lebesgue measure set for every $x\in U$. Therefore, for every $x\in U$, the map $t\mapsto \psi^{m}_{x}(t)$ is a submersion for almost every $t\in (0,T)^{2m}$. The first part of the lemma now follows as above. 
	
	Next, we proceed to prove the `moreover' part. For any bounded, measurable function $f$ on $M$, $x\in U$, writting $t=(t_{2m},\cdots,t_1)$ and $dt=dt_{2m}\cdots dt_1$,
	\begin{align*}
		Q^mf(x)=&\int_{(0,T)^{2m}}\frac{f(\psi_{t}^{m}(x))}{T^{2m}}dt\\
		=&\int_{M}\left(\int_{\{t\in(0,T)^{2m}:\psi_{t}^m(x)=y\}}\frac{f(\psi_{t}^{m}(x))}{T^{2m}|J{\psi_x^{m}}(t)|}\mathcal{H}^{2m-d}(dt)\right)\text{Leb}_M(dy)\\
		=&\int_{M}f(y)\left(\int_{\{t\in(0,T)^{2m}:\psi_{t}^{m}(x)=y\}}\frac{1}{T^{2m}|J{\psi_x^{m}}(t)|}\mathcal{H}^{2m-d}(dt)\right)\text{Leb}_M(dy), 
	\end{align*}
	where the second equality follows from the coarea formula, as can be found, for example, in \cite{FH}. One caveat in our application of the coarea formula is that our assumption says $J{\psi_x^{m}}(t)$ is nonzero only almost-surely. This is not an issue however since $\{y:\psi_{t}^{m}(x)=y\ \text{and}\ J{\psi_x^{m}}(t)=0\}$ has measure zero in $M$ by Sard's lemma.
\end{proof}
Considering any subset $K$ from $(0,T)^{2m}$, any point $x$ in $M$, we define
\begin{equation*}
	A_{K}(x):=\{t\in K: |J{\psi_x^m}(t)|=0\}.
\end{equation*}
\begin{lem}\label{positivity of det J}
	Assume $K$ is a compact set in $(0,T)^{2m}$, $U$ is an open subset of $M$, and the matrix product $D\psi_x^m(t)D\psi_x^m(t)^T$ is almost surely invertible for every $x$ in $U$. Then for any $x_*\in U$ and $\varepsilon>0$ there exists an open set $V\subset (0,T)^{2m}$ and a $\delta>0$ such that 
	\begin{equation*}
		\mathbb{P}(V)<\varepsilon,\ \ B_\delta(x_*)\subset U_0 \ \ \text{and}\ \ A_{K}(x)\subset V\ \text{for any} \ x\in B_\delta(x_*).
	\end{equation*}
	Moreover, there exists an open neighborhood $W$ of $K\cap V^c$ such that 
	\begin{equation*}
		\inf\{|J{\psi_x^m}(t)|:x\in B_\delta(x_*), t\in W\}>0.
	\end{equation*}
\end{lem}
\begin{proof}
	Fix any $x_*\in U$ and any $\varepsilon>0$.  Since the matrix product $D\psi_x^m(t)D\psi_x^m(t)^T$ is almost surely invertible for every $x$ in $U$, $A_{K}(x)$ has Lebesgue measure zero for all $x\in U$. In particular, since $\mathbb{P}$ is absolutely continuous with respect to Lebesgue measure, there exists an open neighborhood $V$ of $A_{K}(x_*)$ such that $\mathbb{P}(V)<\varepsilon$. We will next prove our conclusion using the method of contradiction. Suppose there is a $\delta>0$ such that $B_\delta(x_*)\subset U$ and there exists a sequence $\{x_n\}\subset B_\delta(x_*)$ converging to $x_*$ such that $A_{K}(x_n)$ is not contained in $V$, that is, for each $x_n$ there is a $t_n$ in $K\cap V^c$ satisfying $|J{\Phi_{x_n}^m}(t_n)|=0$. Then since $K\cap V^c$ is compact, there is a subsequence $\{t_{n_k}\}$ which converges to some $t_*\in K\cap V^c$. It follows from the continuous of $(x,t)\mapsto |J{\psi_x^m}(t)|$ that
	\begin{equation*}
		0=|J{\psi_{x_{n_k}}^m}(t_{n_k})|=|J{\psi_{x_*}^m}(t_*)|.
	\end{equation*} 
	But this implies that $t_*$ is in $A_{K}(x_*)$, which contradicts the fact that $A_{K}(x_*)\cap (K\cap V^c)=\emptyset$. 
	
	By the preceding argument and the continuous of $(x,t)\mapsto |J{\psi_x^m}(t)|$, we have
	\begin{equation*}
		\inf\{	|J{\psi_x^m}(t)| : x\in B_{\delta/2}(x_*), \ t\in K\cap V^c\}\geq 2c, 
	\end{equation*}
	for some $c>0$. Set $g_x(t):=|J{\psi_x^m}(t)|$ and let $K'$ be the closure in $\mathbb{R}_+^{mn}$ of 
	\begin{equation*}
		\bigcup_{x\in B_{\delta/2}(x_*)}g_x^{-1}((0,c)).
	\end{equation*}
	Since $K\cap V^c$ and $K'$ are closed and disjoint, they can be seperated by disjoint open sets. By continuity and construction, there exists an open neighborhood $W$ of $K\cap V^c$ satisfies
	\begin{equation*}
		\inf\{|J{\psi_x^m}(t)|:x\in B_{\delta/2}(x_*), t\in W\}\geq c >0.
	\end{equation*}
\end{proof}
\begin{proof}[\textbf{Proof of Theorem \ref{Strong Feller}}]
	Since $\text{Lie}_{x_0}(\mathcal{F})=T_{x_0}M$ for some $x_0\in M$, by the first part of Lemma \ref{submersion for almost all time}, there exist some $m$ and open neighborhood $U$ of $x_0$ such that for every $x\in U$, $D\psi_x^{m}(t)D\psi_x^{m}(t)^T$ is invertible for almost every $t\in(0,T)^{2m}$. Consider a bounded, measurable function $f$ on $M$. The scenario where $f$ is identically zero is straightforward and requires no further discussion. Therefore, let us focus on the alternative case. Specifically, we define the positive constant $C/6$ to be equal to the supremum norm of $f$, denoted as $0 < |f|_{\infty} =: C/6$. Fix any $x_*\in U$ and any $\varepsilon>0$. Since $\mathbb{P}$ is absolutely continuous with respect to Lebesgue measure, there exists a compact subset $K$ of $(0,T)^{2m}$ such that $\mathbb{P}(K^c)<\varepsilon/C$. By Lemma \ref{positivity of det J}, there exists an open set $V\subset (0,T)^{2m}$ and $\delta_1>0$ such that $\mathbb{P}(V)<\varepsilon/C$, $B_{\delta_1}(x_*)\subset U$ and an open neighborhood $W$ of $K\cap V^c$ such that
	\begin{equation}\label{44}
		\inf\{|J\psi_x^{m}(t)| : x\in B_{\delta_1}(x_*), \ t\in W\}=:c>0. 
	\end{equation}
	Since $(0,T)^{2m}=K^c\cup(K\cap V)\cup(K\cap V^c)$, for any $x\in U_0$, we have 
	\begin{align}\label{45}
		Q^{m}f(x)-Q^{m}f(x_*)=&\mathbb{E}\left([f(\psi_{\tau}^{m}(x))-f(\psi_{\tau}^{m}(x_*))]\chi_{K^c}\right)+\mathbb{E}\left([f(\psi_{\tau}^{m}(x))-f(\psi_{\tau}^{m}(x_*))]\chi_{K\cap V}\right)\nonumber\\
		&~~~~~~~+\mathbb{E}\left([f(\psi_{\tau}^{m}(x))-f(\psi_{\tau}^{m}(x_*))]\chi_{K\cap V^c}\right).
	\end{align}
	Due to the selection of $K$, it follows that
	\begin{equation}\label{46}
		|\mathbb{E}\left([f(\psi_{\tau}^{m}(x))-f(\psi_{\tau}^{m}(x_*))]\chi_{K^c}\right)|\leq 2\|f\|_\infty\mathbb{P}(K^c)<\varepsilon/3.
	\end{equation}
	Similarly, the choice of $V$ ensures that
	\begin{equation}\label{47}
		|\mathbb{E}\left([f(\psi_{\tau}^{m}(x))-f(\psi_{\tau}^{m}(x_*))]\chi_{K\cap V}\right)|\leq 2\|f\|_\infty\mathbb{P}(V)<\varepsilon/3.
	\end{equation}
	To handle the term involving $\chi_{K\cap V^c}$, let us define $L^1 := L^1(\text{Leb}_M)$, where $\text{Leb}_M$ denotes the volume form on $M$. We then select a compactly supported, continuous function $\tilde{f}$ on $M$ such that the $L^1$ norm of the difference between $\tilde{f}$ and $f$, $\|\tilde{f} - f\|_{L^1}$, is less than $\sqrt{c}\varepsilon/12$. Note that this can always be done since compactly supported, continuous functions are dense in $L^1$. By adding and subtracting $\tilde{f}\circ\psi_{\tau}^{m}$ appropriately,
	\begin{align}\label{48}
		\mathbb{E}\left([f(\psi_{\tau}^{m}(x))-f(\psi_{\tau}^{m}(x_*))]\chi_{K\cap V^c}\right)=&\mathbb{E}\left([f(\psi_{\tau}^{m}(x))-\tilde{f}(\psi_{\tau}^{m}(x))]\chi_{K\cap V^c}\right)\nonumber\\
		&+\mathbb{E}\left([\tilde{f}(\psi_{\tau}^{m}(x))-\tilde{f}(\psi_{\tau}^{m}(x_*))]\chi_{K\cap V^c}\right)\nonumber\\
		&+\mathbb{E}\left([\tilde{f}(\psi_{\tau}^{m}(x_*))-f(\psi_{\tau}^{m}(x_*))]\chi_{K\cap V^c}\right).
	\end{align} 
	Since $\tilde{f}\circ\psi_{\tau}^{m}$ is continuous, there exists $\delta_2>0$ such that for every $x\in B_{\delta_2}(x_*)$, 
	\begin{equation*}
		|\tilde{f}(\psi_{\tau}^{m}(x))-\tilde{f}(\psi_{\tau}^{m}(x_*))|<\frac{\varepsilon}{6}.
	\end{equation*}
	So for all $x\in B_{\delta_2}(x_*)$, 
	\begin{equation}\label{49}
		\left|\mathbb{E}\left([\tilde{f}(\psi_{\tau}^{m}(x))-\tilde{f}(\psi_{\tau}^{m}(x_*))]\chi_{K\cap V^c}\right)\right|<\frac{\varepsilon}{6}.
	\end{equation}
	Let's focus on the first and third terms on the right side of equation (\ref{48}), writting $t=(t_{2m},\cdots,t_1)$ and $dt=dt_{2m}\cdots dt_1$, it is established that for all $x\in B_{\delta_1}(x_*)$,
	\begin{align}\label{410}
		&\left|\mathbb{E}\left([f(\psi_{\tau}^{m}(x))-\tilde{f}(\psi_{\tau}^{m}(x))]\chi_{K\cap V^c}\right)\right|\nonumber\\
		=&\left|\int_{W} \frac{[f(\psi_{t}^{m}(x))-\tilde{f}(\psi_{t}^{m}(x))]\chi_{K\cap V^c}}{T^{2m}}dt\right|\nonumber\\
		=&\left|\int_{M}\left(\int_{W\cap \{t\in(0,T)^{2m}:\psi_{t}^{m}(x)=\tilde{x}\}}\frac{[f(\psi_{t}^{m}(x))-\tilde{f}(\psi_{t}^{m}(x))]\chi_{K\cap V^c}}{T^{2m}|J{\psi_x^{m}}(t)|}\mathcal{H}^{2m-d}(dt)\right)\text{Leb}_M(d\tilde{x})\right|\nonumber\\
		=&\left|\int_{M}\left(f(\tilde{x})-\tilde{f}(\tilde{x})\right)\left(\int_{W\cap \{t\in(0,T)^{2m}:\psi_{t}^{m}(x)=\tilde{x}\}} \frac{\chi_{K\cap V^c}}{T^{2m} |J{\psi_x^{m}}(t)|}\mathcal{H}^{2m-d}(dt)\right)\text{Leb}_M(d\tilde{x})\right|\nonumber\\
		\leq&\frac{1}{\sqrt{c}}\|f-\tilde{f}\|_{L^1}<\frac{\varepsilon}{12},
	\end{align}
	where the first equality is due to $K\cap V^c$ being a subset of $W$, the second is confirmed by Lemma \ref{submersion for almost all time}, and the first inequality is established through equation (\ref{44}). Subsequently, we set $\delta$ to be the minimum of $\delta_1$ and $\delta_2$, by substituting equations (\ref{49}) and (\ref{410}) into equation (\ref{48}) and considering all $x$ in the ball $B_\delta(x_*)$, we obtain
	\begin{equation}\label{411}
		\left|\mathbb{E}\left([f(\psi_{\tau}^{m}(x))-f(\Phi_{\tau}^{m}(x_*))]\chi_{K\cap V^c}\right)\right|<\frac{\varepsilon}{3}.
	\end{equation}
	Ultimately, substituting Equations (\ref{46}), (\ref{47}) and (\ref{411}) into Equation (\ref{45}) yields,
	\begin{equation*}
		\left|Q^{m}f(x)-Q^{m}f(x_*)\right|<\varepsilon,
	\end{equation*}
	for every $x\in B_\delta(x_*)$. Hence, $Q^mf$ is continuous at every point $x_*$ in $U$, establishing that $Q^{m}$ possesses the strong Feller property on $U$.
\end{proof}

\subsection{Ruling out the degeneracy}
\quad The primary mechanism for ruling out Theorem \ref{degeneracy} $(a)$ is shearing, while Theorem \ref{degeneracy} $(b)$ is ruled out by assumption $(H2)$.
\begin{lem}\label{ruling out a}
	Let $\{\Phi_{\underline{\tau}}^m\}$ be as above. Assume $(H1)$ holds. Then Theorem \ref{degeneracy} $(a)$ does not hold.
\end{lem}
\begin{proof} 
	Since $f_1 \in C^{\omega}(\mathbb{S}^1, \mathbb{R})$, $i = 1, 2$ are non-constant, there exists a point $x_0=(q_0, p_0)\in \text{Supp}(\mu)$ such that $f_1'(p_0)\neq 0$. Given the flow $\varphi_t^{(1)}(x)=(q+t f_1(p),p)$, a direct computation yields
	\[ D\varphi_t(x) = \begin{pmatrix} 
		1 & t f_1'(p) \\ 
		0 & 1 
	\end{pmatrix}.	\]
    Now suppose Theorem \ref{degeneracy} $(a)$ holds, then there exists a Riemannian structure $\{g_x:x\in \text{Supp}(\mu)\}$ on $\text{Supp}(\mu)$ of $\mathcal{O}$ such that (\ref{conformal}) holds for all $m$, $\underline{t}\in[0,T]^{2m}$, $x\in \text{Supp}(\mu)$ and $u,v\in T_x\mathcal{O}$. In particular, 
	\begin{equation*}
		g_{\varphi_t^{(1)}(x_0)}(D\varphi_t^{(1)}(x_0)u,D\varphi_t^{(1)}(x_0)v)=g_{x_0}(u,v),
	\end{equation*}
    for all $t\in [0, T]$ and $u,v\in T_{x_0}\mathcal{O}$. For sufficiently large $T$, there exist distinct $t_1, t_2, t_3\in [0,T]$ such that $\varphi_{t_1}^{(1)}(x_0)=\varphi_{t_2}^{(1)}(x_0)=\varphi_{t_3}^{(1)}(x_0)=x_0$. Choose $v=u$ and write $u=\Big(
	\begin{array}{cc}
		u_1\\
		u_2
	\end{array} \Big)$, then for each $k=1, 2, 3$, we have
    \begin{align*}
    	g_{x_0}(u,u)
    	&= g_{\varphi_{t_k}^{(1)}(x_0)}\left(D\varphi_{t_k}^{(1)}(x_0) u,\ D\varphi_{t_k}^{(1)}(x_0) u\right) \\
    	&= g_{x_0}\left(D\varphi_{t_k}^{(1)}(x_0) u,\ D\varphi_{t_k}^{(1)}(x_0) u\right) \\
    	&= g_{x_0}\left( 
    	\begin{pmatrix} 
    		u_1 + t_k f_1'(p_0) u_2 \\ 
    		u_2 
    	\end{pmatrix}
    	,\ 
    	\begin{pmatrix} 
    		u_1 + t_k f_1'(p_0) u_2 \\ 
    		u_2 
    	\end{pmatrix}
    	\right) \\
    	&= t_k^2 g_{x_0}\left( 
    	\begin{pmatrix} 
    		f_1'(p_0) u_2 \\ 
    		0 
    	\end{pmatrix}
    	,\ 
    	\begin{pmatrix} 
    		f_1'(p_0) u_2 \\ 
    		0 
    	\end{pmatrix}
    	\right) + t_k C_1 + C_2,
    \end{align*}
	where $C_1$ and $C_2$ are independent of $t_k$. Since $f_1'(p_0)\neq 0$, we can choose $u$ such that $f_1'(p)u_2\neq 0$. Consequently, the last equation is a nontrivial quadratic equation in $t_k$ with at most two roots, which contradicts the existence of three distinct values $t_1, t_2, t_3$ satisfying the equation.
\end{proof}

\begin{lem}\label{ruling out b}
	Let $\{\Phi_{\underline{\tau}}^m\}$ be as above. Assume $(H1)$ and $(H2)$ hold. Then Theorem \ref{degeneracy} $(b)$ does not hold.
\end{lem}
\begin{proof}
	It follows from the analyticity of the vector fields that the lifted process 
	\[
	\hat{\Phi}_{\underline{\tau}}^m(x,v)=(\Phi_{\underline{\tau}}^m(x),D_x\Phi_{\underline{\tau}}^m(x)v)
	\]
	is also analytic. Therefore, by Theorem \ref{Strong Feller}, there exists an $m_0$ and a neighborhood $\hat{U}$ of $\hat{x}$ such that $\hat{P}^{m_0}|_{\hat{U}}$ is strong Feller. Assume Theorem \ref{degeneracy} $(b)$ holds, there exist proper linear subspaces $E_x^1,\cdots,E_x^p$ of $T_x \mathcal{O}$ for all $x\in \text{Supp}(\mu)$ such that 
	\begin{equation}\label{permutation}
		D_x\Phi_t^m(E_x^i)=E_{\Phi_t^m(x)}^{\sigma(i)}, \ \ 1\leq i\leq p
	\end{equation}
	for all $m$, $t\in[0,T]^{2m}$ and some permutation $\sigma$. Without loss of generality assume the projection $\pi(\hat{U})$ is contained in $\text{Supp}(\mu)$, by shrinking $\hat{U}$ if necessary. For every $x\in\text{Supp}(\mu)$, define $f:\hat{U}\to \mathbb{R}$ by
	\[f(x,v):=\chi_{E_x}(v)=
	\begin{cases}
		1 & \text{if} \ \ v\in E_x,\\
		0 & \text{otherwise},
	\end{cases}
	\]
	where $E_x=\bigcup_{i=1}^p E_x^i$. Then $f$ is well-defined and is bounded, measurable but discontinuous since $E_x^i$ are proper subspaces. By (\ref{permutation}),
	\begin{align*}
		\hat{P}^{m_0}f(x,v)&=\mathbb{E}\big(f(\Phi_{\underline{\tau}}^{m_0}(x),D_x \Phi_{\underline{\tau}}^{m_0}(x)v)\big)\\
		&=\mathbb{E}\big(\chi_{E_{\Phi_{\underline{\tau}}^{m_0}(x)}}(D_x \Phi_{\underline{\tau}}^{m_0}(x)v)\big)\\
		&=\mathbb{E}\big(\chi_{D_x\Phi_{\underline{\tau}}^{m_0} (E_x)}(D_x \Phi_{\underline{\tau}}^{m_0}(x)v)\big)\\
		&=f(x,v),
	\end{align*}
	is discontinuous which is a contradiction with $\hat{P}^{m_0}f(x,v)$ is continuous since $\hat{P}^{m_0}|_{\hat{U}}$ is strong Feller.
\end{proof}
\begin{proof}[\textbf{Proof of Theorem \ref{Positivity of the top Lyapunov exponent}}]
	Note that $\text{supp}(\mu)=\mathcal{O}$ since $\mu$ is the Lebesgue measure on $\mathcal{O}$. Thus, given assumptions $(H1)$ and $(H2)$, it follows from Lemma \ref{ruling out a} and Lemma \ref{ruling out b} that Theorem \ref{degeneracy} $(a)$ and $(b)$ do not hold. Therefore, $2\lambda_1\neq \lambda_\Sigma=0$, leading to the conclusion that $\lambda_1>0$.
\end{proof}
	
\section{Proof of Theorem \ref{Exponential mixing}}
    \quad In this section, we aim to prove Theorem \ref{Exponential mixing}, which establishes exponential mixing for Hamiltonian shear flow. The proof of this theorem requires proving uniform geometric ergodicity for the two-point chain. To establish uniform geometric ergodicity for the two-point chain, we will use Theorem \ref{Positivity of the top Lyapunov exponent} and the uniform geometric ergodicity for the projective chain. Therefore, we first present the result on uniform geometric ergodicity for the projective chain in the following subsection.
	\subsection{Uniform geometric ergodicity for projective chain}
	\quad We will now demonstrate the irreducibility of the projective Markov chain.
	\begin{lem}\label{irreducibility for projective chain}
		Let $\{\check{\Phi}_{\underline{\tau}}^m\}$ be the projective chain as in Section 2.1. Then, $\{\check{\Phi}_{\underline{\tau}}^m\}$ is exactly controllable.
	\end{lem}
	\begin{proof}
		Consider $q_0$ such that $|f_2(q_0)|$ attains its maximum value and then $f_2'(q_0) = 0$. Similarly, let $p_0$ be such that $|f_1(p_0)|$ attains its maximum value with $f_1'(p_0) = 0$. Furthermore, select $q_1$ such that $f_2(q_1) \neq 0$ and $f_2'(q_1) \neq 0$. Likewise, choose $p_1$ such that $f_1(p_1) \neq 0$ and $f_1'(p_1) \neq 0$.
		Define the sets as follows: $A=\{(q, p, u, v)\in P\mathcal{O}: q=q_1\}$, $B=\{(q, p, u_1, u_2)\in P\mathcal{O}: q=q_0, u_2\neq 0\}$, $C=\{(q_1, p_1, 0, 1), (q_1, p_1, 0, -1)\}$, $D=\{(q_1, p_1, 0, 1)\}$. We are now prepared to demonstrate the following:
		\begin{itemize}
			\item[(1)] $A$ is reachable from $P\mathcal{O}$.\\
			Fix any $(q, p, u_1, u_2)\in P\mathcal{O}$. Since either $f_1(p)\neq 0$ or $f_2(q)\neq 0$, we see that, by taking $\tau_1=0$ and almost any $\tau_2>0$, we may assume without loss of generality that $f_1(p)\neq 0$. Then, by choosing $k\in\mathbb{Z}$ such that $\tau_1=\frac{q_1-q+2k\pi}{f_1(p)}>0$ and choosing $\tau_2=0$, to obtain the conclusion.
			\item[(2)] $B$ is reachable from $A$.\\
			Fix any $(q_1, p, u_1, u_2)\in A$. If $u_2=0$, then by setting $\tau_1=0$, selecting $k_2\in\mathbb{Z}$ such that $\tau_2=\frac{p_1-p+2k_2\pi}{f_2(q_1)}>0$, choosing $k_3\in\mathbb{Z}$ such that $\tau_3=\frac{q_0-q_1+2k_3\pi}{f_1(p_1)}>0$, and setting $\tau_4=0$, we obtain the conclusion. If $u_2\neq 0$, then by selecting $k_2'\in\mathbb{Z}$ such that $\tau_2=\frac{p_1-p+2k_2'\pi}{f_2(q_1)}>0$ and ensuring $f_1(p+\tau_2f_2(q_1))(u_2+\tau_2f_2'(q_1)u_1)\neq 0$. Then, by setting $\tau_1=0$, choosing $k_3'\in\mathbb{Z}$ such that $\tau_3=\frac{q_0-q_1+2k_3'\pi}{f_1(p+\tau_2f_2(q_1))}>0$, and setting $\tau_4=0$, we obtain the conclusion.
			\item[(3)] $C$ is reachable from $B$.\\
			Fix any $(q_0, p, u_1, u_2)\in B$. Choose $\tau_3>0$ and $p_*\in [0,2\pi)$ such that $\tau_3 f_1'(p_*)=-\frac{u_1}{u_2}$ and $\tau_3 f_1(p_*)=q_1-q_0-2k\pi$ for sufficiently large $k>0$. In fact, without loss of generality, we assume there exists $p'$ such that $f_1(p')<0$. Then $\min_{p}f_1(p)<0$, and let $f_1(p_0')=\min_{p}f_1(p)$, then $f_1'(p_0')=0$. For sufficiently small $\varepsilon>0$, we have $f_1'(p_0'-\varepsilon)<0$, $f_1'(p_0'+\varepsilon)>0$, and $f_1(p)<0$ for any $p\in [p_0'-\varepsilon, p_0'+\varepsilon]$. Thus, $\frac{f_1'(p)}{f_1(p)}$ is continuous for any $p\in [p_0'-\varepsilon, p_0'+\varepsilon]$. It follows that 
			\begin{equation*}
				\left[\frac{f_1'(p_0'+\varepsilon)}{f_1(p_0'+\varepsilon)}, \frac{f_1'(p_0'-\varepsilon)}{f_1(p_0'-\varepsilon)}\right]\subseteq \left[\min_{p\in [p_0'-\varepsilon, p_0'+\varepsilon]}\frac{f_1'(p)}{f_1(p)}, \max_{p\in [p_0'-\varepsilon, p_0'+\varepsilon]}\frac{f_1'(p)}{f_1(p)}\right].
			\end{equation*}
			Choose $k>0$ sufficiently large such that
			\begin{equation*}
				-\frac{u_1}{u_2(q_1-q_0-2k\pi)}\in \left[\frac{f_1'(p_0'+\varepsilon)}{f_1(p_0'+\varepsilon)}, \frac{f_1'(p_0'-\varepsilon)}{f_1(p_0'-\varepsilon)}\right].
			\end{equation*}
			By the intermediate value theorem, there exists $p_*\in [p_0'-\varepsilon, p_0'+\varepsilon]$ such that 
			\begin{equation*}
				\frac{f_1'(p_*)}{f_1(p_*)}=-\frac{u_1}{u_2(q_1-q_0-2k\pi)}.
			\end{equation*}
			Set $\tau_3=\frac{q_1-q_0-2k\pi}{f_1(p_*)}$ to obtain $\tau_3 f_1'(p_*)=-\frac{u_1}{u_2}$.
			
			Next, set $\tau_1 = 0$. Choose $k_2 \in \mathbb{Z}$ such that $\tau_2 = \frac{p_* - p + 2k_2\pi}{f_2(q_0)} > 0$, and choose $k_4 \in \mathbb{Z}$ such that $\tau_4 = \frac{p_1 - p_* + 2k_4\pi}{f_2(q_1)} > 0$. Thus, the result is established.
			\item[(4)] $D$ is reachable from $C$.\\
			Fix $(q_1, p_1, 0, -1)\in C$. Choose $q_*$ such that $f_1'(p_1)f_2'(q_*)<0$. Choose $k_1\in\mathbb{Z}$ such that $\tau_1=\frac{q_*-q_1+2k_1\pi}{f_1(p_1)}>0$. Clearly, there exists $M>0$ such that 
			\begin{equation*}
				\frac{(q_*-q_1+2k_1\pi)f_1'(p_1)}{f_1(p_1)}\in[-M, M].
			\end{equation*}
			As in case $(3)$, for sufficiently small $\varepsilon>0$, $f_1(p)<0$ for any $p\in [p_0'-\varepsilon,p_0'+\varepsilon]$, where $p_0'$ is as defined in case $(3)$ and there exists $\delta>0$ such that 
			\begin{equation*}
				[-\delta,\delta]\subseteq \left[\min_{p\in [p_0'-\varepsilon, p_0'+\varepsilon]}\frac{f_1'(p)}{f_1(p)}, \max_{p\in [p_0'-\varepsilon, p_0'+\varepsilon]}\frac{f_1'(p)}{f_1(p)}\right].
			\end{equation*}
			Choose $k_3\in\mathbb{Z}$ such that $\frac{q_1-q_*+2k_3\pi}{f_1(p)}>0$ for any $p\in [p_0'-\varepsilon, p_0'+\varepsilon]$. Choose $k_2\in\mathbb{Z}$ such that $\frac{p-p_1+2k_2\pi}{f_2(q_*)}>0$ sufficiently large for any $p\in [p_0'-\varepsilon, p_0'+\varepsilon]$, and 
			\begin{equation*}
				\left|\left(1+\frac{(q_*-q_1+2k_1\pi)(p-p_1+2k_2\pi)f_1'(p_1)f_2'(q_*)}{f_1(p_1)f_2(q_*)}\right)(q_1-q_*+2k_3\pi)\right|\geq \frac{M}{\delta},
			\end{equation*}  
			for any $p\in [p_0'-\varepsilon, p_0'+\varepsilon]$. Then by the intermediate value theorem, there exists $\tilde{p}\in [p_0'-\varepsilon, p_0'+\varepsilon]$ such that
			\begin{align*}
				&\frac{(q_*-q_1+2k_1\pi)f_1'(p_1)}{f_1(p_1)}\\
				=&\frac{-(q_1-q_*+2k_3\pi)f_1'(\tilde{p})}{f_1(\tilde{p})}\left(1+\frac{(q_*-q_1+2k_1\pi)(\tilde{p}-p_1+2k_2\pi)f_1'(p_1)f_2'(q_*)}{f_1(p_1)f_2(q_*)}\right).
			\end{align*}
			Then set $\tau_2=\frac{\tilde{p}-p_1+2k_2\pi}{f_2(q_*)}$ and $\tau_3=\frac{q_1-q_*+2k_3\pi}{f_1(\tilde{p})}$, choose $k_4\in\mathbb{Z}$ such that $\tau_4=\frac{p_1-\tilde{p}+2k_4\pi}{f_2(q_1)}>0$. This finalizes the argument.
		\end{itemize}
		We conclude by noting that any point $(q,p,u_1,u_2)$ is reachable from a point $(\hat{q},\hat{p}, \hat{u}_1, \hat{u}_2)$ by first traveling to $D$ and then traveling to $(q,p,u_1,u_2)$, using the same arguments as above.
	\end{proof}
	\begin{thm}[\textbf{Uniform geometric ergodicity for projective chain}]\label{Uniformly geometrically ergodic for projective chain}
		Let $\{\check{\Phi}_{\underline{\tau}}^m\}$ be as above. Assume $(H1)$ holds. Then there exists a continuous function $\check{V}:P\mathcal{O}\to [1,\infty)$ such that the transition kernel $\check{P}$ of $\{\check{\Phi}_{\underline{\tau}}^m\}$ is $\check{V}$-uniformly geometrically ergodic.
	\end{thm}
	\begin{proof}
		The Lie brackets of $\check{X_1}$ and $\check{X_2}$ are given by 
		\begin{align*}
			[\check{X_1},\check{X_2}](\check{x})&=D\check{X_2}(\check{x})\check{X_1}(\check{x})-D\check{X_1}(\check{x})\check{X_2}(\check{x})\\
			&=\begin{pmatrix}
				-f_1'(p)f_2(q)\\
				f_1(p)f_2'(q)\\
				-u_1^2u_2f_1(p)f_2''(q)-2u_1u_2^2f_1'(p)f_2'(q)-u_2^3f_2(q)f_1''(p)\\
				u_1^3f_1(p)f_2''(q)+2u_1^2u_2f_1'(p)f_2'(q)+u_1u_2^2f_2(q)f_1''(p)
			\end{pmatrix}
		\end{align*}
		and the following matrix
		\begin{align*}
			&(\check{X_1}(\check{x}),\check{X_2}(\check{x}), [\check{X_1},\check{X_2}](\check{x}))\\
			&=\begin{pmatrix}
				f_1(p) & 0 & -f_1'(p)f_2(q)\\
				0 & f_2(q) & f_1(p)f_2'(q)\\
				u_2^3f_1'(p) & -u_1^2u_2f_2'(q) & -u_1^2u_2f_1(p)f_2''(q)-2u_1u_2^2f_1'(p)f_2'(q)-u_2^3f_2(q)f_1''(p)\\
				-u_1u_2^2f_1'(p) & u_1^3f_2'(q) & u_1^3f_1(p)f_2''(q)+2u_1^2u_2f_1'(p)f_2'(q)+u_1u_2^2f_2(q)f_1''(p)
			\end{pmatrix}
		\end{align*}
		has rank 3 for almost all $\check{x}=(q,p,u_1,u_2)\in P\mathcal{O}$. In fact, let
		\begin{equation*}
			A(\check{x}):=\begin{pmatrix}
				f_1(p) & 0 & -f_1'(p)f_2(q)\\
				0 & f_2(q) & f_1(p)f_2'(q)\\
				u_2^3f_1'(p) & -u_1^2u_2f_2'(q) & -u_1^2u_2f_1(p)f_2''(q)-2u_1u_2^2f_1'(p)f_2'(q)-u_2^3f_2(q)f_1''(p)
			\end{pmatrix}.
		\end{equation*}
		$\det A(\hat{x})$ is a real analytic function and 
		\begin{equation*}
			\det A(q,p,0, 1)=u_2^3f_2^2(q)\left([f_1'(p)]^2-f_1(p)f_1''(p)\right).
		\end{equation*}
	    Since $f_i \in C^{\omega}(\mathbb{S}^1, \mathbb{R})$, $i=1,2$ are non-constant, there exists $q, p$ such that $f_2(q)\neq 0$, $[f_1'(p)]^2-f_1(p)f_1''(p)\neq 0$.
	
		Denote $\check{V}: P\mathcal{O}\to [1,\infty)$ by
		\begin{equation*}
			\check{V}(q,p,u)=V(q,p).
		\end{equation*}
		Then $\check{V}$ satisfies the Lyapunov-Foster drift condition for $\{\check{\Phi}_{\underline{\tau}}^m\}$. Strong aperiodicity follows immediately from $\mathbb{P}(\tau<\varepsilon)>0$ for any $\varepsilon>0$ and $\check{\Phi}_0(x)=x$. By combining Lemma \ref{small set}, and Lemma \ref{irreducibility for projective chain}, we establish that the conditions of Theorem \ref{conditions for uniformly geometrically ergodic} are satisfied. Consequently, the transition kernel $\check{P}$ of $\{\check{\Phi}_{\underline{\tau}}^m\}$ is $\check{V}$-uniformly geometrically ergodic.
	\end{proof}

    \subsection{Uniform geometric ergodicity for two-point chain}
    \quad Our goal in this subsection is to prove uniform geometric ergodicity for the two-point chain. We start by calculating the invariant set of $\tilde{\Phi}_{\underline{\tau}}$ and introducing some notation. Denote
    \[
    \mathcal{I}_1:=\{(q, p, q', p') \in \mathcal{O} \times \mathcal{O}: f_1(p' + t) \equiv f_1(p + t),\,  f_2(q' + t) \equiv f_2(q + t)\, \text{for all } t \in \mathbb{R} \},
    \]
    \[
    \mathcal{I}_2:=\{(q, p, q', p') \in \mathcal{O} \times \mathcal{O}: f_1(p' - t) \equiv f_1(p + t),\,  f_2(q' + t) \equiv -f_2(q + t)\, \text{for all } t \in \mathbb{R} \},
    \]     
    \[
    \mathcal{I}_3:=\{(q, p, q', p') \in \mathcal{O} \times \mathcal{O}: f_1(p' + t) \equiv -f_1(p + t),\,  f_2(q' - t) \equiv f_2(q + t)\, \text{for all } t \in \mathbb{R} \},
    \]   
    and
    \[
    \mathcal{I}_4:=\{(q, p, q', p') \in \mathcal{O} \times \mathcal{O}: f_1(p' - t) \equiv -f_1(p + t),\,  f_2(q' - t) \equiv -f_2(q + t)\, \text{for all } t \in \mathbb{R} \}.
    \]  
    By direct calculation, $\Delta=\cup_{k=1}^4 \mathcal{I}_k$ is shown to be the invariant set under $\tilde{\Phi}_{\underline{\tau}}$. We first prove the irreducibility of the two-point Markov chain.
    \begin{lem}\label{topologically irreducible for two-point chain}
    	Given $(x_*,x_*')\in \mathcal{O}^{(2)}$ and $\varepsilon>0$, $B_\varepsilon(x_*,x_*')$ is reachable from $\mathcal{O}^{(2)}$. In particular, $\{\tilde{\Phi}_{\underline{\tau}}^m\}$ is topologically irreducible.
    \end{lem}
    \begin{proof}
    	Let $x=(q,p), x'=(q',p'), x_*=(q_*,p_*), x_*'=(q_*',p_*')$. Define the sets as follows:
    	\[
    	A=\{(x,x')\in \mathcal{O}^{(2)}: f_1(p) \text{ and } f_1(p') \text{ are rationally linearly independent}\},
    	\]
    	\[
    	B=\{(x,x')\in \mathcal{O}^{(2)}: q\in B_{\varepsilon/3}(q_*), q'\in B_{\varepsilon/3}(q'_*)\},
    	\]
    	\[
    	C=\{(x,x')\in \mathcal{O}^{(2)}: (x,x')\in B_{\varepsilon/3}(B), f_2(q) \text{ and } f_2(q') \text{ are rationally linearly independent}\}.
    	\]
    	To establish that $B_\varepsilon(x_*,x_*')$ is reachable from $\mathcal{O}^{(2)}$, it is sufficient to demonstrate the following:
    	\begin{itemize}
    		\item[(1)] $A$ is reachable from $\mathcal{O}^{(2)}$. Fix any $(x,x')\in \mathcal{O}^{(2)}$ and let $\delta>0$ be arbitrary. Define the projection map $\pi: I\to\mathbb{T}^2$ by $\pi(x,x')=(p, p')$. Define $F_{(x,x')}: [0,\delta)^2\to \mathbb{T}^2$ by 
    		\begin{equation*}
    			F_{(x,x')}(\tau_1,\tau_2)=\pi\circ\tilde{\Phi}_{\underline{\tau}}(x,x')
    			=\begin{pmatrix}
    				p+\tau_2 f_2(q+\tau_1f_1(p))\\
    				p'+\tau_2 f_2(q'+\tau_1f_1(p'))
    			\end{pmatrix}
    		\end{equation*} 
    		Then, 
    		\[
    		D_{(\tau_1,\tau_2)}F_{(x,x')}=
    		\begin{pmatrix}
    			\tau_2 f_2'(q+\tau_1f_1(p))f_1(p)& f_2(q+\tau_1f_1(p)) \\
    			\tau_2 f_2'(q'+\tau_1f_1(p'))f_1(p')& f_2(q'+\tau_1f_1(p'))
    		\end{pmatrix}.
    		\]
    		\begin{itemize}
    			\item[(I)] If $|f_1(p)|\neq |f_1(p')|$, there exists $(\tau_1, \tau_2) \in (0, \delta)^2$ such that $\det D_{(\tau_1,\tau_2)}F_{(x,x')} \neq 0$. Thus, by the constant rank theorem, $A$ is reachable from $\mathcal{O}^{(2)}$. We now proceed to prove this assertion. Without loss of generality, we assume that $f_1(p) \neq 0$. By contradiction, suppose that $\det D_{(\tau_1, \tau_2)} F_{(x,x')} \equiv 0$ for all $(\tau_1, \tau_2) \in (0, \delta)^2$. A direct computation yields that $f_2(q + \tau_1 f_1(p)) \equiv C f_2(q' + \tau_1 f_1(p'))$ for all $\tau_1\in (0,\delta)$, where $C\in\mathbb{R}$ is a constant. If $C=0$, it is evident that $f_2(q + \tau_1 f_1(p)) \equiv 0$ for all $\tau_1\in (0,\delta)$ cannot be true. If $C\neq 0$, we consider the following three cases.
    			\begin{itemize}
    				\item[(A)] $C=\pm1$. When $C=1$, for all $t$ in an open interval $I$ of $\mathbb{R}$, we obtain
    				\begin{equation*}
    					f_2(t) \equiv f_2\left(q'-\frac{qf_1(p')}{f_1(p)}+\frac{tf_1(p')}{f_1(p)}\right).
    				\end{equation*}
    				If $f_1(p') = 0$, $f_2(t) \equiv f_2(q')$ for all $t\in I$, which is a contradiction. If $f_1(p')\neq 0$, without loss of generality, we assume $|f_1(p')/f_1(p)|<1$. 
    				Denote $a=q'-qf_1(p')/f_1(p)$, $b=f_1(p')/f_1(p)$. Substituting into the above equation and iterating gives:
    				\[
    				f_2(t)\equiv f_2(a+bt)\equiv\cdots \equiv f_2\left(\frac{1-b^k}{1-b}a+b^k t \right).
    				\]
    				Let $k\to\infty$, we obtain:
    				\[
    				f_2(t) \equiv f_2\left(\frac{a}{1-b}\right),
    				\]
    				for all $t\in I$, which is a contradiction. The case when $C=-1$ is analogous to that for $C=1$. 
    				\item[(B)] $C\neq \pm 1$. Without loss of generality, we assume $0<|C|<1$. Then for all $t$ in an open interval $I$ of $\mathbb{R}$, we obtain
    				\begin{equation*}
    					f_2(t) \equiv C f_2\left(q'-\frac{qf_1(p')}{f_1(p)}+\frac{tf_1(p')}{f_1(p)}\right).
    				\end{equation*} 
    				Denote $a=q'-qf_1(p')/f_1(p)$, $b=f_1(p')/f_1(p)$. Substituting into the above equation and iterating gives:
    				\[
    				f_2(t)\equiv Cf_2(a+bt)\equiv\cdots \equiv C^kf_2\left(\frac{1-b^k}{1-b}a+b^k t \right).
    				\]
    				Therefore, $|f_2(t)|\leq |C|^k\sup_{x \in S}f_2(x)$ for all $t\in I$. Let $k\to\infty$, we obtain $f_2(t)\equiv 0$ for all $t\in I$, which is a contradiction.
    			\end{itemize}
    			\item[(II)] In the case where $|f_1(p)| = |f_1(p')|$, we divide the analysis into two cases.
    			\begin{itemize}
    				\item[(A)] If $|f_2(q)|\neq |f_2(q')|$, we can choose $\tau_1=0$ and $0 < \tau_2<\delta$ such that $|f_1(p+\tau_2f_2(q))|\neq |f_1(p'+\tau_2f_2(q'))|$. We can then repeat the proof of (1)(I).
    				\item[(B)] If $|f_2(q)| = |f_2(q')|$, we split the analysis into the following cases .
    				\begin{itemize}
    					\item[(i)] If $f_1(p) = f_1(p')\neq 0$, and without loss of generality, assume $f_2(q) = f_2(q')=0$. We proceed by considering two cases. First, if there exists $\tau_1\in(0, \delta)$ such that $|f_2(q+\tau_1f_1(p))|\neq |f_2(q'+\tau_1f_1(p))|$, we can then repeat the proof of (1)(II)(A). Alternatively, if for all $\tau_1 \in (0, \delta)$, we have $|f_2(q+\tau_1f_1(p))|\equiv |f_2(q'+\tau_1f_1(p))|$. Since $f_2\in C^\omega(\mathbb{S}^1,\mathbb{R})$, then either $f_2(q+\tau_1f_1(p))\equiv f_2(q'+\tau_1f_1(p))$ or $f_2(q+\tau_1f_1(p))\equiv -f_2(q'+\tau_1f_1(p))$ for all $\tau_1\in [0,\infty)$. 
    					\begin{enumerate}
    						\item[(a)] If $f_2(q+\tau_1f_1(p))\equiv f_2(q'+\tau_1f_1(p))$ for all $\tau_1\in [0,\infty)$, we claim that there exists $\tau_1, \tau_2\in (0,\delta)$ such that $|f_1(p+\tau_2f_2(q+\tau_1f_1(p)))|\neq |f_1(p'+\tau_2f_2(q'+\tau_1f_1(p)))|$, we can then repeat the proof of (1)(I). By controdiction, assume $|f_1(p+\tau_2f_2(q+\tau_1f_1(p)))|\equiv |f_1(p'+\tau_2f_2(q'+\tau_1f_1(p)))|$ for all 
    						$\tau_1, \tau_2\in (0,\delta)$. Since $f_1(p) = f_1(p')\neq 0$, $f_1(p+\tau_2f_2(q+\tau_1f_1(p))) \equiv f_1(p'+\tau_2f_2(q'+\tau_1f_1(p)))$ for all 
    						$\tau_1, \tau_2\in (0,\delta)$, then $(q,p,q',p')\notin \mathcal{O}^{(2)}$, leading to a contradiction with $(q,p,q',p')\in \mathcal{O}^{(2)}$. 
    						\item[(b)] If $f_2(q+\tau_1f_1(p))\equiv -f_2(q'+\tau_1f_1(p))$ for all $\tau_1\in [0,\infty)$, the proof follows similarly to (II)(B)(i)(a).
    					\end{enumerate}
    					\item[(ii)] If $f_1(p') = -f_1(p)$ and $f_1(p)\neq 0$, the proof follows similarly to (II)(B)(i).
    					\item[(iii)] If $f_1(p')= f_1(p)=0$, since $(q,p,q',p')\notin\mathcal{O}^{(2)}$, then $|f_2(q)| = |f_2(q')| \neq 0$. Without loss of generality, assume that $f_2(q) = f_2(q')\neq 0$, and we proceed by considering two cases. First, if there exists $\tau_2\in(0, \delta)$ such that $|f_1(p+\tau_2f_2(q))|\neq |f_1(p'+\tau_2f_2(q'))|$, we can then repeat the proof of (1)(I). Alternatively, if for all $\tau_2 \in (0, \delta)$, we have $|f_1(p + \tau_2 f_2(q))| \equiv |f_1(p' + \tau_2 f_2(q'))|$. Since $f_1\in C^\omega(\mathbb{S}^1,\mathbb{R})$, then either $f_1(p + \tau_2 f_2(q)) \equiv f_1(p' + \tau_2 f_2(q'))$ or $f_1(p + \tau_2 f_2(q)) \equiv -f_1(p' + \tau_2 f_2(q'))$ for all $\tau_2\in [0,\infty)$. 
    					\begin{enumerate}
    						\item[(a)] If $f_1(p + \tau_2 f_2(q)) \equiv f_1(p' + \tau_2 f_2(q'))$ for all $\tau_2\in [0,\infty)$, we claim that there exists $\tau_2, \tau_3\in (0,\delta)$ such that $|f_2(q+\tau_3f_1(p + \tau_2 f_2(q)))|\neq |f_2(q'+\tau_3f_1(p + \tau_2 f_2(q)))|$, we can then repeat the proof of (1)(II)(A). By controdiction, assume $|f_2(q+\tau_3f_1(p + \tau_2 f_2(q)))|\equiv |f_2(q'+\tau_3f_1(p + \tau_2 f_2(q)))|$ for all 
    						$\tau_2, \tau_3\in (0,\delta)$. If $f_2(q+\tau_3f_1(p + \tau_2 f_2(q)))\equiv f_2(q'+\tau_3f_1(p + \tau_2 f_2(q)))$, then $(q,p,q',p')\notin \mathcal{O}^{(2)}$,  leading to a contradiction with $(q,p,q',p')\in \mathcal{O}^{(2)}$. If $f_2(q+\tau_3f_1(p + \tau_2 f_2(q)))\equiv -f_2(q'+\tau_3f_1(p + \tau_2 f_2(q)))$, then $f_2(q)=-f_2(q')$, a controdiction with $f_2(q)=f_2(q')\neq 0$.
    						\item[(b)] If $f_1(p + \tau_2 f_2(q)) \equiv -f_1(p' + \tau_2 f_2(q'))$ for all $\tau_2\in [0,\infty)$, the proof follows similarly to (II)(B)(i)(a).
    					\end{enumerate}
    				\end{itemize}
    			\end{itemize}
    		\end{itemize}
    		\item[(2)] $B$ is reachable from $A$.\\
    		Fix any $(x,x')\in A$. The flow $(q,q')\mapsto (q+t f_1(p), q'+tf_1(p'))$ on $\mathbb{T}^2$ is dense since $f_1(p)$ and $f_1(p')$ are rationally linearly independent. Then we can choose $\tau_1>0$ and $\tau_2=0$ such that $q+\tau_1 f_1(p)\in B_{\varepsilon/3}(q_*)$ and $q'+\tau_1f_1(p')\in B_{\varepsilon/3}(q_*')$, thus $B$ is reachable from $A$.
    		\item[(3)] $C$ is reachable from $B$.\\
    		This proof is similar to case (1), and by choosing 
    		\[\delta< \frac{\varepsilon}{100M},\]
    		where $M:=\sup_{x \in S}\max\{|f_1(x)|, |f_2(x)|\}$, we deduce that $C$ is reachable from $B$.
    		\item[(4)] $B_\varepsilon(x_*,x_*')$ is reachable from $C$.\\
    		Employing a method of proof similar to that used in case (2), we establish that $B_\varepsilon(x_*,x_*')$ is reachable from $C$.
    	\end{itemize}
    \end{proof}
    For any $(q, p, q', p')$ in $\mathcal{I}_1$, $p'-p$ represents a period of the function $f_1$ and $q'-q$ represents a period of $f_2$. Since $f_i\in C^{\omega}(\mathbb{S}^1,\mathbb{R})$, $i=1, 2$ are non-constant functions, there exist finite constants $a_i, a_j'$ such that for any $(q, p, q', p') \in \mathcal{I}_1$, there exist $i$ and $j$ satisfying $q'-q=a_i$ and $p'-p=a_j'$. Consequently, $\mathcal{I}_1$ can be characterized as the union of the following finitely many invariant sub-manifolds
    \[
    \mathcal{I}_{a_i,a_j'}=\{(q, p, q', p') \in \mathcal{O} \times \mathcal{O}: q'-q=a_i,\, p'-p=a_j'\}.
    \]
    For any $(q, p, q', p')$ in $\mathcal{I}_2$, $(p'+p)/2$ represents the axis of symmetry for the function $f_1$ and $2(q'-q)$ represents a period of $f_2$. Since $f_i\in C^{\omega}(\mathbb{S}^1,\mathbb{R})$, $i=1, 2$ are non-constant functions, there exist finite constants $b_i, b_j'$ such that for any $(q, p, q', p') \in \mathcal{I}_2$, there exist $i$ and $j$ satisfying $2(q'-q)=b_i$ and $(p'+p)/2=b_j'$. Consequently, $\mathcal{I}_2$ can be characterized as the union of the following finitely many invariant sub-manifolds
    \[
    \mathcal{I}_{b_i,b_j'}=\{(q, p, q', p') \in \mathcal{O} \times \mathcal{O}: (p'+p)/2=b_j', \, 2(q'-q)=b_i,\, f_2(q' + t) \equiv -f_2(q + t) \, \text{for all } t \in \mathbb{R}\}.
    \]
    For any $(q, p, q', p')$ in $\mathcal{I}_3$, $2(p'-p)$ represents a period of $f_1$ and $(q'+q)/2$ represents the axis of symmetry for the function $f_2$. Since $f_i\in C^{\omega}(\mathbb{S}^1,\mathbb{R})$, $i=1, 2$ are non-constant functions, there exist finite constants $c_i, c_j'$ such that for any $(q, p, q', p') \in \mathcal{I}_3$, there exist $i$ and $j$ satisfying $2(p'-p)=c_j'$ and $(q'+q)/2=c_i$. Consequently, $\mathcal{I}_3$ can be characterized as the union of the following finitely many invariant sub-manifolds
    \[
    \mathcal{I}_{c_i,c_j'}=\{(q, p, q', p') \in \mathcal{O} \times \mathcal{O}: (q'+q)/2=c_i, \, 2(p'-p)=c_j',\, f_1(p' + t) \equiv -f_2(p + t) \, \text{for all } t \in \mathbb{R}\}.
    \]
    For any $(q, p, q', p')$ in $\mathcal{I}_4$, the graph of $f_2$ is symmetric about the point $\left((q'+q)/2,0\right)$ and the graph of $f_1$ is symmetric about the point $\left((p'+p)/2,0\right)$. Since $f_i\in C^{\omega}(\mathbb{S}^1,\mathbb{R})$, $i=1, 2$ are non-constant functions, there exist finite constants $d_i, d_j'$ such that for any $(q, p, q', p') \in \mathcal{I}_4$, there exist $i$ and $j$ satisfying $(q'+q)/2=d_i$ and $(p'+p)/2=d_j'$. Consequently, $\mathcal{I}_4$ can be characterized as the union of the following finitely many invariant sub-manifolds
    \[
    \mathcal{I}_{d_i,d_j'}=\{(q, p, q', p') \in \mathcal{O} \times \mathcal{O}: (q'+q)/2=d_i, \, (p'+p)/2=d_j'\}.
    \]
    Next, we prove the Lyapunov-Foster drift condition for the two-point chain $\{\tilde{\Phi}_{\underline{\tau}}^m\}$. Without loss of generality, we only need to consider $\Delta=\cup_{k=0}^4 I_k$, where $I_0=\{(x,y)\in \mathcal{O} \times \mathcal{O}: y=x\}$, $I_1=\mathcal{I}_{a_1,a_1'}$, $I_2=\mathcal{I}_{b_1,b_1'}$ and $I_3=\mathcal{I}_{c_1,c_1'}$, $I_4=\mathcal{I}_{d_1,d_1'}$. For $s>0$, we define 
    \[
    I_k(s):=\left\{(x,y)\in\mathcal{O}^{(2)}: d(x(k),y)<s\right\},\quad k=0, 1, 2, 3, 4
    \]
    where $x=(q, p)$, $x(0)=x$, $x(1)=(a_1+q, a_1'+p) \mod 2\pi\mathbb{Z}^2$, $x(2)=(b_1/2+q, 2b_1'-p)\mod 2\pi\mathbb{Z}^2$, $x(3)=(2c_1-q, c_1'/2+p)\mod 2\pi\mathbb{Z}^2$ and $x(4)=(2d_1-q, 2d_1'-p)\mod 2\pi\mathbb{Z}^2$. First, we consider the case where $(x,y) \in I_0(s_*)$ with $s^*>0$ to be determined. We need the following results from Proposition 4.5 in \cite{MT}. 
	\begin{lem}\cite[Proposition 4.5]{MT}\label{W}
		Assume that the one-point and projective chains are $V$-uniformly geometrically ergodic and $\check{V}$-uniformly geometrically ergodic, respectively, and furthermore, that the Lyapunov exponent $\lambda_1$ of $\{\Phi_{\underline{\tau}}^m\}$ is positive. Then there exist $h, s_*>0$, $0<\gamma<1$ and a continuous function $w:\mathcal{O}^{(2)}\to (0,\infty)$ satisfying $0<c\leq w\leq C<\infty$, such that $W_0(x,y):=d(x,y)^{-h}w(x,y)$ satisfies the following:
		\begin{equation*}
			P^{(2)}W_0(x,y)<\gamma W_0(x,y),
		\end{equation*} 
	for all $(x,y)\in I_0(s_*)$.
	\end{lem}
    Let $s_*$ be the value specified in Lemma \ref{W}. For any $(x,y) \in \mathcal{O}^{(2)}$ and for $k=1, 2, 3, 4$, define
    \begin{equation*}
    	W_k(x,y)= W_0(x(k),y). 
    \end{equation*}
    From Lemma \ref{W}, it follows that
    \begin{equation*}
    	P^{(2)}W_k(x,y)=P^{(2)}W_0(x(k),y)< \gamma W_0(x(k),y) =\gamma W_k(x,y),
    \end{equation*}
    for all $(x,y) \in I_k(s_*)$. Define the function $W: \mathcal{O}^{(2)}\to (0,\infty)$ by
    \begin{equation*}
    	W(x,y)=\max_{k}W_k(x,y).
    \end{equation*}
    Set $\Delta(s_*)=\cup_{k=0}^4I_k(s_*)$. Then, for all $(x,y)\in \Delta(s_*)$, we have
    \begin{equation*}
    	P^{(2)}W(x,y)\leq \gamma W(x,y).
    \end{equation*}
    As established in Proposition \ref{Lyapunov-Foster drift condition}, the function $V$ on $\mathcal{O}$ is given by
    \begin{equation*}
    	V(q,p)=\max(d(q,C_{f_2}),d(p,C_{f_1}))^{-\beta}+b
    \end{equation*}
    We now introduce a function $V_1$, defined as $V_1(q,p)=V(q,p)-b$. With the results above, we proceed to prove the Lyapunov-Foster drift condition for the two-point chain, which is inspired by \cite[Proposition 18]{Co}.
    \begin{pro}\label{VV}
    	Let $\{\tilde{\Phi}_{\underline{\tau}}^m\}$ be as above. Assume $(H1)$ and $(H2)$ hold. Then there exists an integrable function $V^{(2)}:\mathcal{O}^{(2)}\to [1,\infty)$ such that the transition kernel $P^{(2)}$ of $\{\tilde{\Phi}_{\underline{\tau}}^m\}$ is $V^{(2)}$-uniformly geometrically ergodic. 
    \end{pro}
    \begin{proof}
    	Let $h$, $s_*$, $w$ and $\gamma$ be given by Lemma \ref{W}. Without relabeling, let $\gamma<1$ be large enough to satisfy the drift condition inequality for $V$, so $\gamma\geq \alpha$, as in Proposition \ref{Lyapunov-Foster drift condition}. Let $\varepsilon_0$ be as in the proof of Proposition \ref{Lyapunov-Foster drift condition}. Define
    	\[
    	C:=\{(x,y)\in \mathcal{O}^{(2)}: (x,y)\notin \Delta(s_*)\text{ and } \text{dist}(\{x,y\}, F)\geq \varepsilon\},
    	\]
    	where $F:=\{(q,p): f_1(p)=0 \text{ and } f_2(q)=0\}$ and $\varepsilon>0$ will be chosen later in the proof. For sufficiently small $a>0$, define 
    	\begin{equation*}
    		V^{(2)}(x,y)= \chi_{\mathcal{O}^{(2)}\setminus C}(x,y)(W(x,y)+a(V_1(x)+V_1(y))) + c_0\chi_C(x,y),
    	\end{equation*}
    	where $\chi_D$ is a characteristic function on $D$ and $c_0\geq 1$ is a constant.

    	To start, let $\delta\in (0,\varepsilon_0)$ be choosen small enough such that $\text{dist}(\Phi_{\underline{\tau}}(x), F)<\delta$ implies $\text{dist}(x, F)<s_*/2$ for any $\underline{\tau}\in [0,T]^2$, and let $\varepsilon\leq \delta$. We need to show that there exists $0<\tilde{\gamma}<1$ such that
    	\[
    	\mathbb{E}[V^{(2)}(\tilde{\Phi}_{\underline{\tau}}(x,y))]\leq \tilde{\gamma} V^{(2)}(x,y),
    	\]
    	for any $(x,y)\in {\mathcal{O}^2\setminus C}$. To establish this result, we analyze the following three cases.
    	\begin{itemize}
    		\item[(a)] $\text{dist}(x,F)<\delta$ and $\text{dist}(y, F)<\delta$, then $d(x,y)<s_*$. Denote 
    		\[
    		\tilde{\gamma}:=\frac{\gamma+1}{2}.
    		\]
    		Then, we have
    		\[
    		\mathbb{E}[W(\tilde{\Phi}_{\underline{\tau}}(x,y))+a(V_1(\Phi_{\underline{\tau}}(x))+V_1(\Phi_{\underline{\tau}}(y)))]\leq \tilde{\gamma} (W(x,y)+a(V_1(x)+V_1(y))).
    		\]
    		\item[(b)] $\text{dist}(x,F)\geq \delta$ and $\text{dist}(y, F)\geq\delta$, since $(x,y)\in {\mathcal{O}^2\setminus C}$, then $(x,y)\in \Delta(s_*)$.
    		Denote
    		\[
    		d_0:=\min\{\text{dist}(\Phi_{\underline{\tau}}(z),F): \underline{\tau}\in [0,T]^2 \text{ and } \text{dist}(z,F)\geq \delta\}.
    		\]
    		Clearly, $0<d_0\leq \delta$.
    		Denote
    		\[
    		a:=\frac{\tilde{\gamma}-\gamma}{2}d_0^\beta s_*^{-h}\min w,
    		\]
    		where $\beta$ is given in the proof of  Proposition \ref{Lyapunov-Foster drift condition}.  Then we have 
    		\begin{align*}
    			&\mathbb{E}[W(\tilde{\Phi}_{\underline{\tau}}(x,y))+a(V_1(\Phi_{\underline{\tau}}(x))+V_1(\Phi_{\underline{\tau}}(y)))]\\
    			\leq&\gamma W(x,y)+2ad_0^{-\beta}\\
    			\leq &\tilde{\gamma}W(x,y)+(\gamma-\tilde{\gamma})W(x,y)+2ad_0^{-\beta}\\
    			\leq &\tilde{\gamma}W(x,y)+\tilde{\gamma} aV_1(x)+\tilde{\gamma} aV_1(y).
    		\end{align*}
    	    \item[(c)] Without loss of generality, we consider $\text{dist}(x,F)<\delta$, $\text{dist}(y, F)\geq\delta$. Now we split into two sub-cases:
    	    \begin{itemize}
    	    	\item[(i)] $(x,y)\in \Delta(s_*)$. Then, we have 
    	    	\begin{align*}
    	    		&\mathbb{E}[W(\tilde{\Phi}_{\underline{\tau}}(x,y))+a(V_1(\Phi_{\underline{\tau}}(x))+V_1(\Phi_{\underline{\tau}}(y)))]\\
    	    		\leq&\gamma W(x,y)+\gamma a V_1(x)+ad_0^{-\beta}\\
    	    		\leq &\tilde{\gamma}W(x,y)+\gamma a V_1(x)+(\gamma-\tilde{\gamma})W(x,y)+ad_0^{-\beta}\\
    	    		\leq &\tilde{\gamma}W(x,y)+\tilde{\gamma} aV_1(x)+\tilde{\gamma} aV_1(y).
    	    	\end{align*}
    	    	\item[(ii)] $(x,y)\notin \Delta(s_*)$, since $(x,y)\in {\mathcal{O}^2\setminus C}$, then $\text{dist}(x,F)<\varepsilon$. Let $\Phi_{\underline{\tau}}(x) = (q_1, p_1)$. We define 
    	    	\[
    	    	\Phi_{\underline{\tau}}(x)(0) = \Phi_{\underline{\tau}}(x),
    	    	\]
    	    	\[
    	    	\Phi_{\underline{\tau}}(x)(1) = (a_1 + q_1, a_1' + p_1) \mod 2\pi \mathbb{Z}^2,
    	    	\]
    	    	\[
    	    	\Phi_{\underline{\tau}}(x)(2) = \left( \frac{b_1}{2} + q_1, 2b_1' - p_1 \right) \mod 2\pi \mathbb{Z}^2,
    	    	\]
    	    	\[
    	    	\Phi_{\underline{\tau}}(x)(3) = \left(2c_1 - q_1, \frac{c_1'}{2} + p_1 \right) \mod 2\pi \mathbb{Z}^2,
    	    	\]
    	    	\[
    	    	\Phi_{\underline{\tau}}(x)(4) = (2d_1 - q_1, 2d_1' - p_1) \mod 2\pi \mathbb{Z}^2.
    	    	\]
    	    	Define the quantity $d_2$ as 
    	    	\[
    	    	d_2:=\min\{d(\Phi_{\underline{\tau}}(x)(k),\Phi_{\underline{\tau}}(y)): 0\leq k\leq 4, \ \underline{\tau}\in [0,T]^2 \text{ and } (x,y)\notin \Delta(s_*)\}.
    	    	\]
    	       Choose
    	    	\[
    	    	\varepsilon:=\left(\frac{(\tilde{\gamma}-\gamma)a}{d_2^{-h}\max w+ad_0^{-\beta}}\right)^{\frac{1}{\beta}}.
    	    	\]
    	    	Then, we have
    	    	\begin{align*}
    	    		&\mathbb{E}[W(\tilde{\Phi}_{\underline{\tau}}(x,y))+a(V_1(\Phi_{\underline{\tau}}(x))+V_1(\Phi_{\underline{\tau}}(y)))]\\
    	    		\leq&\mathbb{E}[W(\tilde{\Phi}_{\underline{\tau}}(x,y))+\gamma a V_1(x)+a\mathbb{E}(V_1(\Phi_{\underline{\tau}}(y)))\\
    	    		\leq &d_2^{-h}\max w+\tilde{\gamma} aV_1(x)+(\gamma-\tilde{\gamma})aV_1(x)+ad_0^{-\beta}\\
    	    		\leq &\tilde{\gamma} aV_1(x)+(\tilde{\gamma}-\gamma)a \varepsilon^{-\beta}+(\gamma-\tilde{\gamma})aV_1(x)\\
    	    		\leq &\tilde{\gamma}W(x,y)+\tilde{\gamma} aV_1(x)+\tilde{\gamma} aV_1(y).
    	    	\end{align*}
    	    \end{itemize}
    	\end{itemize}
    \end{proof}
    \begin{thm}[\textbf{Uniform geometric ergodicity for two-point chain}]\label{Uniformly geometrically ergodic for two-point chain}
    	Let $\{\tilde{\Phi}_{\underline{\tau}}^m\}$ be as above. Assume $(H1)$, $(H2)$ and $(H3)$ hold. Then there exists a continuous function $V^{(2)}:\mathcal{O}^{(2)}\to [1,\infty)$ such that the transition kernel $P^{(2)}$ of $\{\tilde{\Phi}_{\underline{\tau}}^m\}$ is $V^{(2)}$-uniformly geometrically ergodic.
    \end{thm}
    \begin{proof}
	    Strong aperiodicity follows immediately from $\mathbb{P}(\tau<\varepsilon)>0$ for any $\varepsilon>0$ and $\tilde{\Phi}_0(x)=x$. By combining Lemma \ref{small set}, Lemma \ref{topologically irreducible for two-point chain}, and Proposition \ref{VV}, we establish that the conditions of Theorem \ref{conditions for uniformly geometrically ergodic} are satisfied. Consequently, the transition kernel $P^{(2)}$ of $\{\tilde{\Phi}_{\underline{\tau}}^m\}$ is $V^{(2)}$-uniformly geometrically ergodic.
    \end{proof}

  \subsection{Exponential mixing}
	\quad Our goal in this subsection is to prove Theorem \ref{Exponential mixing} assuming $(H1)$, $(H2)$ and $(H3)$. The proof of Theorem \ref{Exponential mixing} uses the following result, which is inspired by \cite[Proposition 5]{Co}.
	\begin{pro}\label{sup mix time}
		Let $\{\Phi_{\underline{\tau}}^m\}$ be as above. Assume $(H1)$, $(H2)$ and $(H3)$ hold. Then for every $g\in L^\infty (\mathbb{T}^{2})$ with mean zero, there exists a random variable $\eta>0$ which is finite almost surely, such that for every ball $B\subseteq \mathbb{T}^{2}$, we have
		\[
		\sup\left\{m\in\mathbb{N}: \left|\frac{1}{|B|}\int_B g(\Phi_{\underline{\tau}}^m(x))dx\right|>1\right\}\leq \eta(\underline{\tau})|\log(r(B))|.
		\]
		for every $\underline{\tau}$.
	\end{pro}
	\begin{proof}
		It follows from Theorem \ref{Uniformly geometrically ergodic for two-point chain} that $P^{(2)}$ is $V^{(2)}$-uniformly geometrically ergodic, $P^{(2)}$ admits a unique stationary measure $\mu^{(2)}$ on $\mathcal{O}^{(2)}$, and there exist constants $C>0$ and $\lambda\in (0,1)$ such that the bound 
		\begin{equation*}
			\left |\left(P^{(2)}\right)^m g^{(2)}(x,y)-\int g^{(2)} d\mu^{(2)}\right |\leq CV^{(2)}(x,y)\|g^{(2)}\|_{V^{(2)}}\lambda^m,
		\end{equation*}
		holds for all $(x,y)\in \mathcal{O}^{(2)}$ and $g^{(2)}\in \mathcal{M}_{V^{(2)}}(\mathcal{O}^{(2)})$. Given any $g\in L^\infty (\mathbb{T}^{2})$ with mean zero, take $g^{(2)}(x,y)=g(x)g(y)$, we have
		\begin{equation*}
			\left|\mathbb{E}g(\Phi_{\underline{\tau}}^m(x))g(\Phi_{\underline{\tau}}^m(y))\right|\leq C V^{(2)}(x,y)\|g\|_{L^\infty}^2\lambda^m.
		\end{equation*}
		By Chebyshev's inequality and Fubini's theorem, for any ball $B\subseteq \mathbb{T}^{2}$, we have 
		\begin{align*}
			\mathbb{P}\left\{\left|\frac{1}{|B|}\int_B g(\Phi_{\underline{\tau}}^m(x))dx\right|>\frac{1}{3}\right\}&\leq \frac{9}{|B|^2}\mathbb{E}\left|\int_B g(\Phi_{\underline{\tau}}^m(x))dx\right|^2\\
			&\leq \frac{C\|g\|_{L^\infty}^2\lambda^m}{|B|^2}\int_B\int_B V^{(2)}(x,y)dxdy.
		\end{align*}
		Notationally, we write $C>0$ to denote a constant which may change from line to line. Covering $\mathbb{T}^{2}$ by at most $Cr^{-2}$ many balls of radius $r/2$, we conclude from the union bound that 
		\begin{equation*}
			\mathbb{P}\left\{\exists B, r(B)>\frac{r}{2}\text{ and } \left|\frac{1}{|B|}\int_B g(\Phi_{\underline{\tau}}^m(x))dx\right|>\frac{1}{2}\right\}\leq \frac{C\|g\|_{L^\infty}^2\lambda^m}{r^{6}}\int_{\mathcal{O}^{(2)}}V^{(2)}(x,y)dxdy.
		\end{equation*}
		Using the union bound over all $r=1/i$ for $i\geq 2$, we have
		\begin{align*}
			&\mathbb{P}\left\{\exists B,\ \exists m\geq k|\log(r(B))| \text{ and } \left|\frac{1}{|B|}\int_B g(\Phi_{\underline{\tau}}^m(x))dx\right|>1\right\}\\
			\leq&\sum_{i=2}^{\infty}\sum_{m\geq k\log i}\mathbb{P}\left\{\exists B, r(B)>\frac{1}{i} \text{ and } \left|\frac{1}{|B|}\int_B g(\Phi_{\underline{\tau}}^m(x))dx\right|>\frac{1}{2}\right\}\\
			\leq & C\sum_{i=2}^{\infty}\sum_{m\geq k\log i} i^{6}\lambda^m\leq \frac{C}{1-\lambda}(-7-k\log\lambda)^{-1}.
		\end{align*}
		Define
		\begin{equation*}
			\eta(\underline{\tau}):=\sup_{m, B}\left\{\frac{m}{|\log(r(B))|}: \left|\frac{1}{|B|}\int_B g(\Phi_{\underline{\tau}}^m(x))dx\right|>1\right\}.
		\end{equation*}
		Then we have 
		\begin{align*}
			\mathbb{P}\left\{\eta(\underline{\tau})=\infty\right\}&=\lim_{k\to\infty}\mathbb{P}\left\{\eta\geq k\right\}\\
			&= \lim_{k\to\infty}\mathbb{P}\left\{\exists B,\ \exists m\geq k|\log(r(B))| \text{ and } \left|\frac{1}{|B|}\int_B g(\Phi_{\underline{\tau}}^m(x))dx\right|>1\right\}\\
			&\leq \lim_{k\to\infty}\frac{C}{1-\lambda}(-7-k\log\lambda)^{-1}=0,
		\end{align*}
		so $\eta>0$ is finite almost surely.
	\end{proof}
	\begin{proof}[\textbf{Proof of Theorem \ref{Exponential mixing}}]
		For any initial data $u_0\in L^\infty(\mathbb{T}^{2})$ with mean zero, by Proposition \ref{sup mix time}, there is a random variable $\eta>0$ which is finite almost surely and satisfies 
		\begin{equation*}
			\sup\left\{m\in\mathbb{N}: \frac{1}{|B|}\left|\int_B u_{\underline{\tau}}(m,x)dx\right|>1\right\}\leq \eta(\underline{\tau})|\log (r(B))|,
		\end{equation*}
		for all balls $B\subseteq \mathbb{T}^{2}$ and almost every $\underline{\tau}$. For any integer $\varrho>0$, by the definition of $\text{mix}(b_{\underline{\tau}}^\varrho)$, there exist a sequence of $B_j\subseteq \mathbb{T}^{2}$ such that $r(B_j)\geq \text{mix}(b_{\underline{\tau}}^\varrho)-1/j$ and 
		\[
		\frac{1}{|B_j|}\left|\int_{B_j}u_{\underline{\tau}}(\varrho,x)dx\right|>1.
		\]
		Thus we have $\eta(\underline{\tau})|\log (r(B_j))|\geq \varrho$. Taking the limit on both sides with respect to $j$, we obtain 
		\[\eta(\underline{\tau})\left|\log \text{mix}(b_{\underline{\tau}}^\varrho)\right|\geq\varrho.\]
		And since $\|D_xb_{\underline{\tau}}^\varrho\|_{L^1([0,1]\times \mathbb{T}^{2})}\leq CT\varrho$ for almost every $\underline{\tau}$, it follows that 
		\begin{equation*}
			\left|\log \text{mix}(b_{\underline{\tau}}^\varrho)\right|\geq \xi(\underline{\tau})\|D_xb_{\underline{\tau}}^\varrho\|_{L^1([0,1]\times \mathbb{T}^{2})},
		\end{equation*}
		where $\xi^{-1}(\underline{\tau})=CT\eta(\underline{\tau})$.

	\end{proof}

	\section{Proofs of Corollary \ref{The Pierrehumbert model with randomized time} and Corollary \ref{Analog model to the Chirikov standard map}}
	\quad In this section, we aim to prove Corollary \ref{The Pierrehumbert model with randomized time}, concerning exponential mixing for the Pierrehumbert model with randomized time, as well as Corollary \ref{Analog model to the Chirikov standard map}, which demonstrates exponential mixing for an analogous model to the Chirikov standard map.
	\begin{proof}[\textbf{Proof of Corollary \ref{The Pierrehumbert model with randomized time}}]
	 Let $f_1(p)=\sin p$ and $f_2(q)=\sin q$ be as defined in Section 1, we have $\mathcal{O}=\mathbb{T}^2\setminus\{0,\pi\}^2$. It is also clear that $C_{f_i}\cap C_{f_i'}=\emptyset$, so assumption $(H1)$ holds. Consequently, Theorem \ref{Uniformly geometrically ergodic} applies to this model.
	
	By a straightforward computation, the matrix 
	\begin{equation*}
		A(\hat{x}):=(\hat{X}_1(\hat{x}),\hat{X}_2(\hat{x}), [\hat{X}_1,\hat{X}_2](\hat{x}),[\hat{X}_1,[\hat{X}_1,\hat{X}_2]](\hat{x}) )
	\end{equation*}
	of the lifted vector fields 
	\begin{equation*}
		\hat{X}_1(\hat{x}):=(\sin p, 0, v\cos p, 0),\ \ \hat{X}_2(\hat{x}):=(0, \sin q, 0, u\cos q), 
	\end{equation*}
	and its Lie bracket 
	\begin{equation*}
		[\hat{X}_1,\hat{X}_2](\hat{x})=D\hat{X}_2(\hat{x})\hat{X}_1(\hat{x})-D\hat{X}_1(\hat{x})\hat{X}_2(\hat{x})=\begin{pmatrix}
			-\cos p\sin q,\\
			\cos q \sin p\\
			-u\cos p \cos q+v\sin p \sin q\\
			v\cos p\cos q- u\sin p\sin q 
		\end{pmatrix},
	\end{equation*}
	\begin{equation*}
		[\hat{X}_1,[\hat{X}_1,\hat{X}_2]](\hat{x})=\begin{pmatrix}
			-2 \cos p \cos q \sin p\\
			-\sin^2 p\sin q\\
			-2v\cos^2 p \cos q+2v\cos q \sin^2 p+u\sin 2p\sin q\\
			-\sin p(u\cos q\sin p+2v\cos p\sin q) 
		\end{pmatrix}
	\end{equation*}
	has rank $4$ for almost all $\hat{x}=(q,p,u,v)$ since $\det A(\hat{x})$ is a real analytic function and 
	\begin{equation*}
		\det A(\pi/3, \pi/3, \sqrt{2}/2, \sqrt{2}/2)=\frac{27}{128}\neq 0.
	\end{equation*}
	Then assumption $(H2)$ is satisfied. Consequently, Theorem \ref{Positivity of the top Lyapunov exponent} applies to this model.
	
	In this model, $\mathcal{O}^{(2)}:=\mathcal{O}\times \mathcal{O}\setminus \Delta$, where
	\[
	\Delta= \left\{ (q,p,q',p') \in \mathcal{O}\times \mathcal{O} \left| 
	\begin{aligned}
		(q', p') &= (q, p), \\
		\text{or }  (q', p')&= (\pi + q, \pi - p) \mod 2\pi\mathbb{Z}^2, \\
		\text{or } (q', p') &= (\pi - q, \pi + p) \mod 2\pi\mathbb{Z}^2, \\
		\text{or } (q', p') &= (2\pi - q, 2\pi - p) \mod 2\pi\mathbb{Z}^2
	\end{aligned}
	\right.
	\right\}.
	\]
	For any $(x,y)=(q,p,q_1,p_1)\in\mathcal{O}^{(2)}$,
	\begin{equation*}
		\tilde{X}_1(\hat{x}):=(\sin p, 0, \sin p_1, 0),\ \ \hat{X}_2(\hat{x}):=(0, \sin q, 0, \sin q_1), 
	\end{equation*}
    and its Lie bracket 
    \begin{align*}
    	[\tilde{X}_1,\tilde{X}_2](x,y)&=D\tilde{X}_2(x,y)\tilde{X}_1(x,y)-D\tilde{X}_1(x,y)\tilde{X}_2(x,y)\\
    	&=\begin{pmatrix}
    		-\sin q \cos p\\
    		\cos q\sin p\\
    		-\sin q_1 \cos p_1\\
    		\cos q_1\sin p_1
    	\end{pmatrix},
    \end{align*}
    \begin{equation*}
    	[[\tilde{X}_1,\tilde{X}_2],\tilde{X}_1](x,y)=\begin{pmatrix}
    		\cos q \sin(2p)\\
    		\sin q\sin^2 p\\
    		\cos q_1 \sin(2p_1)\\
    		\sin q_1\sin^2 p_1
    	\end{pmatrix}.                                                                      
    \end{equation*}
    Then, the matrix 
    \begin{equation*}
    	B(x,y):=(\tilde{X}_1(x,y),\tilde{X}_2(x,y), [\tilde{X}_1,\tilde{X}_2](x,y),[[\tilde{X}_1,\tilde{X}_2],\tilde{X}_1](x,y) )
    \end{equation*}
    has rank $4$ for almost all $(x,y)$ since $\det B(x,y)$ is a real analytic function and 
    \begin{equation*}
    	\det B(\pi/3, \pi/2, \pi/2, \pi/3)=\frac{\sqrt{3}}{2}\neq 0.
    \end{equation*}
    Then assumption $(H3)$ is satisfied. Consequently, Theorem \ref{Exponential mixing} applies to this model.
	\end{proof}

	\begin{proof}[\textbf{Proof of Corollary \ref{Analog model to the Chirikov standard map}}]
		Let $f_1(p)=\sin p$ and $f_2(q)=q$ be as defined in Section 1, we have $\mathcal{O}=\mathbb{T}^2\setminus \{(0,0), (0,\pi)\}$. It is also clear that $C_{f_i}\cap C_{f_i'}=\emptyset$, so assumption $(H1)$ holds. Consequently, Theorem \ref{Uniformly geometrically ergodic} applies to this model.
		
		By a straightforward computation, the matrix 
		\begin{equation*}
			A(\hat{x}):=(\hat{X}_1(\hat{x}),\hat{X}_2(\hat{x}), [\hat{X}_1,\hat{X}_2](\hat{x}),[\hat{X}_1,[\hat{X}_1,\hat{X}_2]](\hat{x}) )
		\end{equation*}
		of the lifted splitting vector fields 
		\begin{equation*}
			\hat{X}_1(\hat{x}):=(\sin p, 0, v\cos p, 0),\ \ \hat{X}_2(\hat{x}):=(0, q, 0, u), 
		\end{equation*}
		and its Lie bracket 
		\begin{equation*}
			[\hat{X}_1,\hat{X}_2](\hat{x})=\begin{pmatrix}
				-q\cos p\\
				\sin p\\
				-u\cos p+q v\sin p\\
				v\cos p
			\end{pmatrix},\quad
			[\hat{X}_1,[\hat{X}_1,\hat{X}_2]](\hat{x})=\begin{pmatrix}
				-2\cos p\sin p\\
				0\\
				-2v\cos(2p)\\
				0
			\end{pmatrix}
		\end{equation*}
		has rank $4$ for almost all $\hat{x}=(q,p,u,v)$ since $\det A(\hat{x})$ is a real analytic function and 
		\begin{equation*}
			\det A(\pi/3, \pi/3, \sqrt{2}/2, \sqrt{2}/2)=\frac{9-\sqrt{3}\pi}{16}\neq 0.
		\end{equation*}
		Then assumption $(H2)$ is satisfied. Consequently, Theorem \ref{Positivity of the top Lyapunov exponent} applies to this model.
		
		In this model, $\mathcal{O}^{(2)}:=\mathcal{O}\times \mathcal{O}\setminus \{(x,x):x\in\mathcal{O}\}$. For any $(x,y)=(q,p,q_1,p_1)\in\mathcal{O}^{(2)}$,
		\begin{equation*}
			\tilde{X}_1(\hat{x}):=(\sin p, 0, \sin p_1, 0),\ \ \hat{X}_2(\hat{x}):=(0, q, 0, q_1), 
		\end{equation*}
		and its Lie bracket 
		\begin{equation*}
			[\tilde{X}_1,\tilde{X}_2](x,y)=\begin{pmatrix}
				-q \cos p\\
				\sin p\\
				-q_1 \cos p_1\\
				\sin p_1
			\end{pmatrix},\quad
			[\tilde{X}_1,[\tilde{X}_1,\tilde{X}_2]](x,y)=\begin{pmatrix}
				-\sin(2p)\\
				0\\
				-\sin(2p_1)\\
				0
			\end{pmatrix}.                                                                      
		\end{equation*}
		Then, the matrix 
		\begin{equation*}
			B(x,y):=(\tilde{X}_1(x,y),\tilde{X}_2(x,y), [\tilde{X}_1(x,y),\tilde{X}_2(x,y)],[\tilde{X}_1,[\tilde{X}_1,\tilde{X}_2]](x,y))
		\end{equation*}
		has rank $4$ for almost all $(x,y)$ since $\det B(x,y)$ is a real analytic function and 
		\begin{equation*}
			\det B(\pi, \pi/2, \pi/2, \pi/3)=\frac{(3-\sqrt{3})\pi}{4}\neq 0.
		\end{equation*}
		Then assumption $(H3)$ is satisfied. Consequently, Theorem \ref{Exponential mixing} applies to this model.
	\end{proof}

\end{document}